\documentclass[11pt, a4paper]{amsart}
\usepackage{a4wide}
\usepackage{graphicx,enumitem}
\usepackage{amsmath,amssymb}
\usepackage[english]{babel}
\usepackage{subfigure}
\usepackage{units}
\usepackage{diagbox}
\usepackage{xcolor}
\usepackage{caption}
\usepackage{ulem}
\usepackage{cancel}
\usepackage{scalerel,stackengine}
\DeclareSymbolFontAlphabet{\amsmathbb}{AMSb}%
\usepackage{mathbbol}
\usepackage{latexsym}

\usepackage[pdftex,
pdftitle={Approximation of SPDE covariance by finite elements},
bookmarksopen,
colorlinks,
linkcolor=black,
urlcolor=black,
citecolor=black
]{hyperref}
\hypersetup{pdfauthor={M. Kov\'acs, A. Lang, and A. Petersson}}

\usepackage{dsfont}

\usepackage{mathtools}

\DeclarePairedDelimiter{\floor}{\lfloor}{\rfloor}

\newcommand{\R}{\amsmathbb{R}}

\newcommand{\C}{\amsmathbb{C}}
\newcommand{\bH}{\amsmathbb{H}}
\newcommand{\N}{\amsmathbb{N}}

\newcommand{\cA}{\mathcal{A}}
\newcommand{\cB}{\mathcal{B}}
\newcommand{\cC}{\mathcal{C}}
\newcommand{\cD}{\mathcal{D}} 

\newcommand{\cF}{\mathcal{F}}

\newcommand{\cH}{\mathcal{H}}

\newcommand{\cK}{\mathcal{K}}
\newcommand{\cL}{\mathcal{L}}

\newcommand{\cO}{\mathcal{O}}

\newcommand{\cV}{\mathcal{V}}

\DeclareMathOperator{\E}{\amsmathbb{E}} %
\DeclareMathOperator{\Cov}{\mathsf{Cov}}

\DeclareMathOperator{\trace}{Tr}

\DeclareMathOperator{\im}{\mathrm{im}}
\DeclareMathOperator{\ke}{\mathrm{ker}}

\newcommand{\dd}{\,\mathrm{d}}
\newcommand{\dom}{\mathrm{dom}} %

\newcommand{\inpro}[3][{}]{ \langle #2 , #3 \rangle_{#1} }
\newcommand{\norm}[2]{\| #1 \|_{#2}}
\newcommand{\lrnorm}[2]{\left\| #1 \right\|_{#2}}
\newcommand{\Bignorm}[2]{\Big\| #1 \Big\|_{#2}}

\newcommand{\dfloor}[1]{\floor{#1}_{\Delta t}}

\stackMath
\newcommand\reallywidehat[1]{%
	\savestack{\tmpbox}{\stretchto{%
			\scaleto{%
				\scalerel*[\widthof{\ensuremath{#1}}]{\kern-.6pt\bigwedge\kern-.6pt}%
				{\rule[-\textheight/2]{1ex}{\textheight}}%
			}{\textheight}%
		}{0.5ex}}%
	\stackon[1pt]{#1}{\tmpbox}%
}
\newtheorem{lemma}{Lemma}[section]
\newtheorem{proposition}[lemma]{Proposition}

\newtheorem{theorem}[lemma]{Theorem}

\theoremstyle{remark}
\newtheorem{remark}[lemma]{Remark}

\theoremstyle{definition}

\newtheorem{assumption}[lemma]{Assumption}
\newtheorem{example}[lemma]{Example}

\definecolor{darkgreen}{rgb}{0,.6,0}

\begin{document}
	\title[Approximation of SPDE covariances by finite elements]{Approximation of SPDE covariance operators by finite elements: A semigroup approach
	}
	
	\author[M.~Kov\'acs]{Mih\'aly Kov\'acs} \address[Mih\'aly Kov\'acs]{\newline Faculty of Natural Sciences, Department of Differential Equations 
	\newline Budapest University of Technology and Economics
	\newline M\H{u}egyetem rkp. 3., H-1111 Budapest, Hungary, 
	\newline Faculty of Information Technology and Bionics
	\newline P\'azm\'any P\'eter Catholic University
	\newline H-1444 Budapest, P.O. Box 278, Hungary
	\newline and
	\newline Department of Mathematical Sciences
	\newline Chalmers University of Technology \& University of Gothenburg
	\newline S--412 96 G\"oteborg, Sweden.} \email[]{kovacs.mihaly@itk.ppke.hu}
	
	\author[A.~Lang]{Annika Lang} \address[Annika Lang]{\newline Department of Mathematical Sciences
		\newline Chalmers University of Technology \& University of Gothenburg
		\newline S--412 96 G\"oteborg, Sweden.} \email[]{annika.lang@chalmers.se}
	
	\author[A.~Petersson]{Andreas Petersson} \address[Andreas Petersson]{\newline The Faculty of Mathematics and Natural Sciences
		\newline Department of Mathematics
		\newline Postboks 1053
		\newline Blindern
		\newline 0316 Oslo, Norway.}
	 \email[]{andreep@math.uio.no}
	
	\thanks{The authors wish to express many thanks to two anonymous reviewers who helped to improve the results and presentation. M.\ Kov\'acs acknowledges the support of the Marsden Fund of the Royal Society of New Zealand	through grant. no. 18-UOO-143, the Swedish Research Council (VR) through project no.\ 2017-04274 and the National Research, Development, and Innovation Fund of Hungary through grant no.\ 131545 and TKP2021-NVA-02. The work of A.\ Lang was partially supported by the Swedish Research Council (VR) (project no.\ 2020-04170), by the Wallenberg AI, Autonomous Systems and Software Program (WASP) funded by the Knut and Alice Wallenberg Foundation, and by the Chalmers AI Research Centre (CHAIR). The work of A. Petersson was supported in part by the Research Council of Norway (RCN) through project no.\ 274410, the Swedish Research Council (VR) through reg.~no.~621-2014-3995 and the Knut and Alice Wallenberg foundation.}
	
	
	\begin{abstract}
		The problem of approximating the covariance operator of the mild solution to a linear stochastic partial differential equation is considered. An integral equation involving the semigroup of the mild solution is derived and a general error decomposition  is proven. This formula is applied to approximations of the covariance operator of a stochastic advection-diffusion equation and a stochastic wave equation, both on bounded domains. The approximations are based on finite element discretizations in space and rational approximations of the exponential function in time. Convergence rates are derived in the trace class and Hilbert--Schmidt norms with numerical simulations illustrating the results.  
	\end{abstract}
	
	\maketitle
	
	\section{Introduction}
	
	This paper considers stochastic partial differential equations (SPDEs) 
	formulated as linear stochastic evolution equations on a Hilbert space $H$. That is to say, equations of the form
	\begin{equation}
	\label{eq:introSPDE}
	\begin{split}
	\dd X(t) + A X(t) \dd t &= F X(t) \dd t + B \dd W(s) \text{ for } t \in (0,T], T < \infty, \\
	X(0) & = \xi.
	\end{split}
	\end{equation}
	Here $X$ is an $H$-valued stochastic process, $F$ and $B$ are linear operators, $W$ is a Wiener process in $H$ with covariance operator $Q$ and $-A$ is the generator of a $C_0$-semigroup $S = (S(t))_{t \in [0,T]}$ of linear operators on $H$.  Since~\eqref{eq:introSPDE} is linear and the noise term is additive and Gaussian, the solution $X(t)$ to~\eqref{eq:introSPDE} at time $t \in (0,T]$ is an $H$-valued Gaussian random variable when the initial value $\xi$ is Gaussian. The distribution of $X(t)$ is therefore completely determined by its mean value $\E[X(t)]$ and covariance (operator) $\Cov(X(t)) = \E[(X(t) - \E[X(t)]) \otimes (X(t) - \E[X(t)])]$. Computing these quantities is therefore vital for understanding of~$X(t)$. In general, there are no analytic solutions, so numerical approximations are needed. In this paper, we focus on the approximation of $\Cov(X(t))$. 
	
	The literature on the numerical analysis of approximations of SPDE covariance operators is sparse. We are only aware of~\cite{K20,KLL17,LLS13}. Therein, the authors consider SPDEs of parabolic type and solve a tensorized equation related to the concept of a weak solution to~\eqref{eq:introSPDE}. %
	They assume the operator $A$ to be self-adjoint. In this paper, we take a different approach. We work with the mild solution of an SPDE to derive an operator-valued integral equation for the covariance, expressed in terms of the semigroup $S$. 
	
	This approach allows us to treat parabolic SPDEs where $A$ is not self-adjoint. One example of such an equation comes from the modeling of sea surface temperature dynamics, see \cite{HH87}. The equation, posed in the space $H = L^2(\cD)$ of square integrable functions on some domain $\cD \subset \R^2$, is given by
	\begin{equation}	
	\label{eq:advdiff-intro}	
	\dd X(t,x) + \cA X(t,x) \dd t= \dd W(t,x) \text{ for } t \in (0,T], x \in \cD.
	\end{equation}
	Here $A X(t)$ corresponds to $\cA X(t,\cdot) = \delta X(t,\cdot)-D \Delta X(t,\cdot) -\mathbf{a}\cdot\nabla X(t,\cdot) $, where  $X(t,x)$ is the sea surface temperature at time $t$ and point $x \in \cD$, $\mathbf{a} \in \R^2$ is the velocity vector field of the upper ocean layer, $D > 0$ is a diffusion coefficient and $\delta \in \R$ is a feedback parameter. When $\mathbf{a} \neq 0$, the operator $A$ ceases to be self-adjoint. The Wiener process $W$ models small time scale fluctuations in the heat flux across the ocean-atmosphere interface. Its covariance operator $Q$ is an integral operator with a kernel $q$ having small correlation length. Similar models have recently been considered for reconstructing the evolution of cloud systems from discrete measurements, see~\cite{MSM17}. %
	
	Moreover, our approach extends the parabolic setting to any SPDE which has a mild solution in terms of a semigroup. This includes hyperbolic SPDEs, such as a stochastic equation for the vertical displacement $U(t,x)$ of a strand of DNA suspended in a liquid at time~$t$ and space $x \in \cD \subset \R^d$, $d = 1, 2, 3$ from \cite{D09}. It is given by
	\begin{equation}
	\label{eq:wave-intro}
	\dd \dot{U} (t,x) - \Delta U(t,x) \dd t =  - (Q U)(t,x) \dd t + \dd W(t,x) \text{ for } t \in (0,T], x \in \cD.
	\end{equation}
	The first term on the right hand side models friction due to viscosity of the fluid, while the Wiener process term $W$ corresponds to random bombardment of the DNA strand by the fluid's molecules. Writing $X = [U,\dot{U}]^\top$, the equation can be put in the form of~\eqref{eq:introSPDE}	by considering it on a product space, see~Section~\ref{sec:wave}. Figure~\ref{fig:path} shows realizations of the solutions to~\eqref{eq:advdiff-intro} and~\eqref{eq:wave-intro} for the domain $\cD = (0,1)$, see Examples~\ref{ex:heat} and~\ref{ex:wave}. %
	\begin{figure}[ht!]
		\centering
		\subfigure[A stochastic advection-diffusion equation. \label{subfig:advdiff-path}]{\includegraphics[width = .49\textwidth]{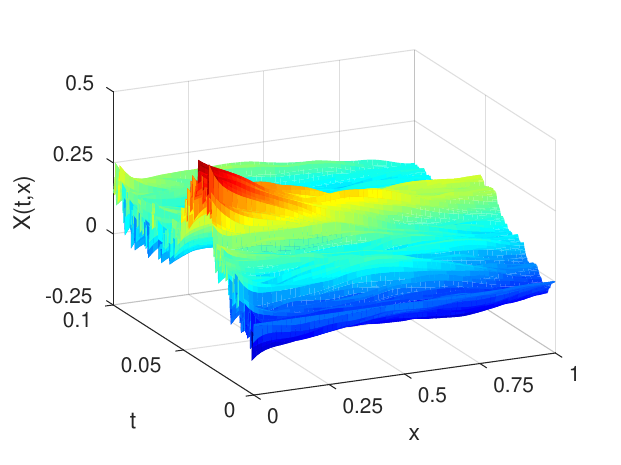}}
		\subfigure[A stochastic wave equation. \label{subfig:wave-path}]{\includegraphics[width = .49\textwidth]{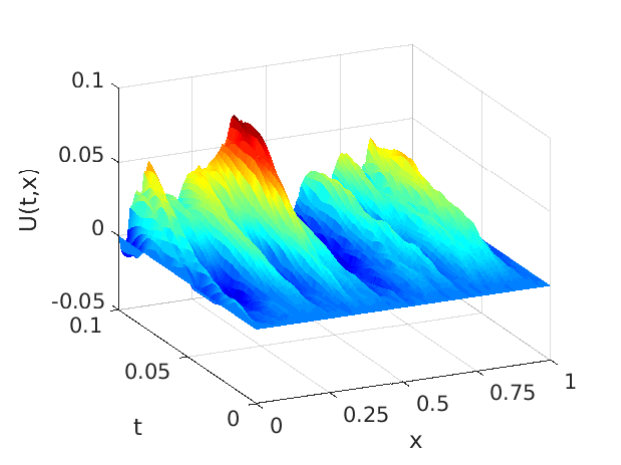}}
		\caption{Realizations of the solutions $X$ and $U$ to~\eqref{eq:advdiff-intro} and~\eqref{eq:wave-intro} when $\cD = (0,1)$.}
		\label{fig:path}
	\end{figure} 
	
	Let us now outline the content of the paper. In Section~\ref{sec:covariance}, we formulate an operator-valued integral equation for the covariance of the mild solution $X$ to~\eqref{eq:introSPDE} in an abstract Hilbert space framework, and give assumptions that ensure that it has a unique solution. We confirm that $X$ is Gaussian and that the process $[0,T] \ni t \mapsto \Cov(X(t))$ is a solution to the integral equation. The mild It\^o formula in \cite{DPJR19} is key for this. We finish the section by giving an abstract error decomposition formula for approximations of this process. The error is, for $t \in [0,T]$, analyzed with respect to the norms $\norm{\cdot}{\cL_1(H)}$ and~$\norm{\cdot}{\cL_2(H)}$. Here $\cL_1(H)$ and $\cL_2(H)$ denote the spaces of trace class and Hilbert--Schmidt operators, respectively. The first norm is a natural choice since if $(Q_j)_{j=1}^\infty$ is a sequence of covariances of some Gaussian $H$-valued random variables $(X_j)_{j=1}^\infty$ with zero mean, then $Q_j \to \Cov(X(t))$ in $\cL_1(H)$ if and only if $X_j \to X(t)$ weakly, i.e., $\E[f(X_j)] \to \E[f(X(t))]$ for all continuous and bounded functionals $f$ on $H$, see \cite{C83}. The norm of $\cL_2(H)$ is weaker. It has a natural meaning when $H = L^2(\cD)$: if $X(t) = X(t,\cdot)$ is $\cF \times \cB(\cD)$-measurable, %
	\begin{equation*}
	\inpro[H]{\Cov(X(t)) u}{v} = \int_{\cD \times \cD} \Cov(X(t,x),X(t,y)) u(x) v(y) \dd x \dd y 
	\end{equation*}
	and $\norm{\Cov(X(t))}{\cL_2(H)}^2 = \norm{\Cov(X(t,\cdot),X(t,\cdot))}{L^2(\cD \times \cD)}^2$.
	Therefore we may, formally at least, view convergence in $\cL_2(H)$ as convergence in $L^2(\cD \times \cD)$ of underlying covariance functions on $\cD$. Figure~\ref{fig:cov} shows the covariance functions for the solutions to~\eqref{eq:advdiff-intro} and~\eqref{eq:wave-intro} at $T=0.1$.
	\begin{figure}[ht!]
		\centering
		\subfigure[Covariance function for a stochastic advection-diffusion equation. \label{subfig:advdiff-cov}]{\includegraphics[width = .49\textwidth]{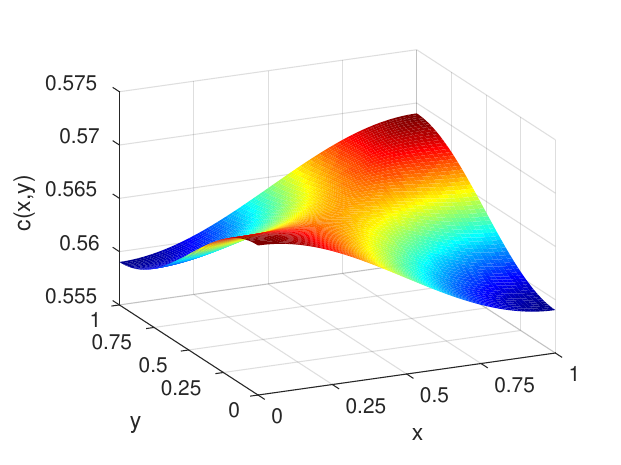}}
		\subfigure[Covariance function for a stochastic wave equation. \label{subfig:wave-cov}]{\includegraphics[width = .49\textwidth]{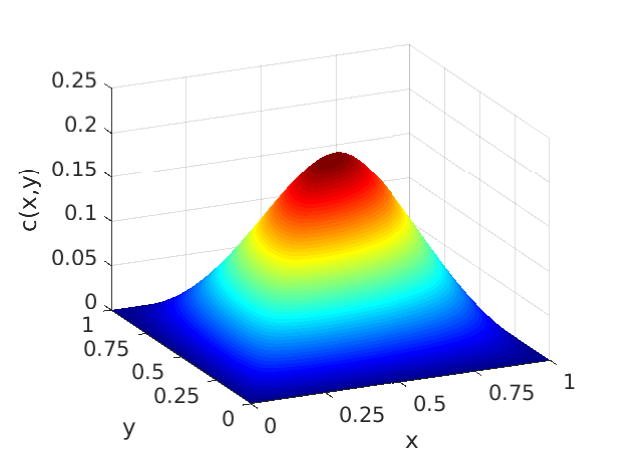}}
		\caption{Plot of the covariance functions $c(x,y) = \Cov(X(0.1,x),X(0.1,y))$ and $c(x,y) = \Cov(U(0.1,x),U(0.1,y))$, $x,y \in \cD = (0,1)$, for the solutions $X$ and $U$ to~\eqref{eq:advdiff-intro} and~\eqref{eq:wave-intro}.}
		\label{fig:cov}
	\end{figure}
	
	In Section~\ref{sec:applications} we apply our abstract framework to the two concrete equations~\eqref{eq:advdiff-intro} and~\eqref{eq:wave-intro}. In both cases, the covariance integral equations are discretized by finite elements in space and rational approximations of the driving semigroup in time. The resulting approximations are expressed as integral equations based on a fully discrete approximation $\tilde S$ of $S$.
	
	Consider a discrete approximation $\tilde X(t)$ of $X(t)$, $t \in [0,T]$. Suppose, without loss of generality, that $\E[X(t)] = \E[\tilde X(t)] = 0$. Then, by properties of the norms $\norm{\cdot}{\cL_i(H)}$, $i \in \{1,2\}$, and the H\"older inequality,
	\begin{equation}
	\label{eq:introcovboundedbystrong}
	\begin{split}
	&\norm{\Cov(X(t))-\Cov(\tilde X(t))}{\cL_i(H)} \\ 
	&\quad= \norm{\E[X(t) \otimes X(t)] - \E[\tilde X(t) \otimes \tilde X(t)]}{\cL_i(H)} \\ 
	&\quad= \frac{1}{2} \Bignorm{\E\left[(X(t) + \tilde X(t)) \otimes (X(t) - \tilde X(t))\right] \hspace{-1pt} + \hspace{-1pt}\E\left[(X(t) - \tilde X(t)) \otimes (X(t) + \tilde X(t))\right]}{\cL_{i}(H)} \\
	&\quad\le \max\left(\E[\norm{X(t)}{H}^2]^{\frac{1}{2}},\E[\norm{\tilde X(t)}{H}^2]^{\frac{1}{2}}\right) \E\left[\norm{X(t)-\tilde X(t)}{H}^2\right]^{\frac{1}{2}}.
	\end{split}
	\end{equation}
	Therefore, if $\Cov(\tilde X(t))$ can be calculated, we obtain an approximation scheme for $\Cov(X(t))$ for which the error can be bounded by the strong error $\E\left[\norm{X(t)-\tilde X(t)}{H}^2\right]^{1/2}$. In Section~\ref{sec:heat}, where a stochastic advection-diffusion equation is considered, we demonstrate that this bound is suboptimal. More precisely, we compare the convergence rate for a covariance approximation based on replacing $S$ by a fully discrete approximation $\tilde S$ in the operator-valued integral equation with the rate for the strong error with respect to an approximation $\tilde X$ of $X$ based on the same discretization $\tilde S$. It turns out that the convergence rate for our approximation is typically higher than the strong error, see Remark~\ref{rem:heatstrongvcov}. We also demonstrate, in theory and by a numerical simulation, that there are cases when the $\cL_2$ error decays faster than the stronger $\cL_1$ error.
	
	In Section~\ref{sec:wave} we work with the stochastic wave equation and provide, to the best of our knowledge, the first results on convergence rates for the approximation of covariance operators of hyperbolic SPDEs. We first consider a temporally semidiscrete approximation, which is then used for analyzing a fully discrete approximation. Numerical simulations finish the section.
	
	Throughout the paper, we adopt the notion of generic constants, which may vary from occurrence to occurrence and are independent of any parameter of interest, such as spatial or temporal step sizes. By $a \lesssim b$ we denote the existence of a generic constant such that $a \le C b$.
	
	\section{Covariance operators of stochastic evolution equations}
	\label{sec:covariance}
	
	In this section we prove existence and uniqueness of the solution to an operator-valued integral equation. We then show that the covariance of the mild solution of an SPDE is a solution of this equation. First, however, we introduce our setting and reiterate some facts from operator theory and probability theory in Hilbert spaces. 
	
	\subsection{Operator theory}
	
	Let $(H, \inpro[H]{\cdot}{\cdot})$ and $(U, \inpro[U]{\cdot}{\cdot})$ be real and separable Hilbert spaces. We denote by $\cL(H,U)$ the space of bounded linear operators from $H$ to $U$ equipped with the usual operator norm and by $\cL_1(H,U)$ and $\cL_2(H,U)$ the spaces of trace class and Hilbert--Schmidt operators, respectively. We use the shorthand notations $\cL(H)$, $\cL_1(H)$ and $\cL_2(H)$ when $U=H$. Additionally, we denote by $\Sigma(H) \subset \cL(H)$ the set of symmetric bounded operators on $H$. An operator $\Gamma \in \cL_1(H,U)$ if and only if there are two orthonormal sequences $(e_j)_{j=1}^\infty \subset H$, $(f_j)_{j=1}^\infty \subset U$ and a sequence $(\mu_j)_{j=1}^\infty \in \ell^1$ such that 
	\begin{equation}
	\label{eq:trace_characterization}
	\Gamma x = \sum_{j = 1}^{\infty} \mu_j \inpro{x}{e_j} f_j  \text{ for } x \in H.
	\end{equation} 
	The space $\cL_1(H,U)$ is a separable Banach space with norm
	\begin{equation}
	\label{eq:trace_norm}
	\norm{\Gamma}{\cL_1(H,U)} = \inf_{\substack{(a_j) \subset H \\ (b_j) \subset U}} \left\{ \sum_{j = 1}^{\infty} \norm{a_j}{H} \norm{b_j}{U} : \Gamma = \sum_{j = 1}^{\infty} \inpro[H]{\cdot}{a_j} b_j \right\},
	\end{equation}
	where the infimum is taken over all sequences $(a_j)_{j=1}^\infty \subset H, (b_j)_{j=1}^\infty \subset U$ see~\cite[Sections~47-48]{T67b}. Moreover, $\cL_2(H,U)$ is a separable Hilbert space with an inner product, for an arbitrary orthonormal basis $(e_j)_{j = 1}^\infty$ of $H$, given by $$\inpro[\cL_2(H,U)]{\Gamma_1}{\Gamma_2} = \sum_{j = 1}^{\infty} \inpro{\Gamma_1 e_j}{\Gamma_2 e_j}\text{ for }\Gamma_1, \Gamma_2 \in \cL_2(H,U),$$
	We have $\cL_1(H,U) \subset \cL_2(H,U)$ and  $\Gamma \in \cL_i(H,U)$ if and only if $\Gamma^* \in \cL_i(U,H)$ with 
	\begin{equation}
	\label{eq:schattenadjoint}
	\norm{\Gamma}{\cL_i(H,U)} = \norm{\Gamma^*}{\cL_i(U,H)}, \, i \in \{1,2\}.
	\end{equation}
	We identify the spaces $\cL_2(H,U)$ and $U \otimes H$, the Hilbert tensor product, with equivalent norms. The tensor $u \otimes v$ is regarded as an element of $\cL(H,U)$ by the relation $(u \otimes v) w = \inpro[H]{v}{w} u$ for $v,w \in H$ and $u \in U$. It can be seen that $u \otimes v \in \cL_1(H,U)$ with $\norm{u \otimes v}{\cL_1(H,U)} = \norm{u \otimes v}{\cL_2(H,U)} = \norm{u}{U} \norm{v}{H}$. Moreover, \begin{equation}
	\label{eq:HSinpro}
	\inpro[\cL_2(H,U)]{\Gamma}{u \otimes v} = \inpro[U]{\Gamma v}{u} \text{ for }\Gamma \in \cL_2(H,U).
	\end{equation} 
	If $V$ and $G$ are two other real and separable Hilbert spaces, then 
	\begin{equation}
	\label{eq:HSlinops}
	\Gamma_1 u \otimes \Gamma_2 v = \Gamma_1 (u \otimes v) \Gamma_2^*
	\end{equation}
	for $u \in U$, $v \in H$, $\Gamma_1 \in \cL(U,V)$ and $\Gamma_2 \in \cL(H,G)$, with $\Gamma_2^*$ denoting the adjoint of $\Gamma_2$. 
	The spaces $\cL_i(H,U)$, $i \in \{1,2\}$, are operator ideals: if $\Gamma_1 \in \cL(G,V)$, $\Gamma_2 \in \cL_i(U,G)$ and $\Gamma_3 \in \cL(H,U)$ then $\Gamma_3 \Gamma_2 \Gamma_1 \in \cL_i(H,V)$ with 
	\begin{equation}
	\label{eq:schatten_bound_1}
	\norm{\Gamma_1 \Gamma_2 \Gamma_3}{\cL_i(H,V)} \le \norm{\Gamma_1}{\cL(G,V)} \norm{\Gamma_2}{\cL_i(U,G)} \norm{\Gamma_3}{\cL(H,U)}.
	\end{equation}
	If $\Gamma_1 \in \cL_2(U,V)$ and $\Gamma_2 \in \cL_2(H,U)$, then $\Gamma_1 \Gamma_2 \in \cL_1(H,V)$ and
	\begin{equation}
	\label{eq:schatten_bound_2}
	\norm{\Gamma_1 \Gamma_2}{\cL_1(H,V)} \le \norm{\Gamma_1}{\cL_2(U,V)} \norm{\Gamma_2}{\cL_2(H,U)}. 
	\end{equation}
	The trace of $\Gamma \in \cL_1(H)$ is, for an arbitrary orthonormal basis $(e_j)_{j=1}^\infty$ of $H$, defined by $$\trace(\Gamma) = \sum^\infty_{j=1} \inpro[H]{\Gamma e_j}{e_j}.$$ If $\Gamma \in \Sigma^+(H) \subset \Sigma(H)$, the space of all positive semidefinite operators, then $\trace(\Gamma) = \norm{\Gamma}{\cL_1(H)}$. 
	
	\subsection{Probability theory in Hilbert spaces}
	
	Below we work on the bounded interval $[0,T]$, $T < \infty$. Let $(\Omega, \cA, (\cF_t)_{t \in [0,T]}, P)$ be a complete filtered probability space satisfying the usual conditions, which is to say that $\cF_0$ contains all $P$-null sets and $\cF_t = \cap_{s > t} \cF_s$ for all $t \in [0,T]$. 	
	By $L^p(\Omega,H)$, $p \in [1, \infty)$	we denote the space of all $H$-valued random variables~$Y$ with norm $\norm{Y}{L^p(\Omega,H)} = (\E[\norm{Y}{H}^p])^{1/p}$. For $Y,Y' \in L^2(\Omega,H)$, the cross-covariance (operator) of $Y,Y'$ is defined by $\Cov(Y,Y') = \E[(Y - \E[Y]) \otimes (Y' - \E[Y'])] \in \cL_1(H)$
	and the covariance of $Y$ by $\Cov(Y) = \Cov(Y,Y) \in \cL_1(H) \cap \Sigma^+(H)$. 
	Note that $\Cov(Y,Y')$ is uniquely determined by $\inpro[H]{\Cov(Y,Y') u}{v} = \Cov(\inpro[H]{Y}{v},\inpro[H]{Y'}{u})$. 
	
	An $H$-valued random variable $Y$ is said to be Gaussian if $\inpro[H]{Y}{v}$ is a Gaussian real-valued random variable for all $v \in H$. Then $Y \in L^p(\Omega,H)$ for all $p \ge 1$. A pair $Y,Y'$ of $H$-valued random variables is said to be jointly Gaussian if $Y \oplus Y'$ is an $H \oplus H$-valued Gaussian random variable. Then $Y$ and $Y'$ are independent if and only if $\Cov(Y,Y') =0$, cf.~\cite{M84}. 
	
	Let $W$ be a generalized Wiener process in $U$ (see \cite[Chapter~4]{DPZ14}) with covariance $Q \in \Sigma^+(U)$, not necessarily of trace class. 
	The Hilbert space $(Q^{1/2}(U), \inpro[U]{Q^{-1/2}\cdot}{Q^{-1/2}\cdot})$ is denoted by $U_0$, where $Q^{1/2}$ is the unique positive semidefinite square root of $Q$ and $Q^{-1/2}$ its pseudoinverse. 
	
	\subsection{Covariance integral equations and mild solutions to SPDEs}
	Let $\hat{H}$ be another Hilbert space such that $H \hookrightarrow \hat{H}$ continuously and densely. The main topic of study in this paper are integral equations of the form
	\begin{equation}
	\label{eq:C}
	\begin{split}
	K(t) &= S(t) Q_\xi S(t)^* + \int_{0}^{t}  S(t-s) F K(s)  S(t-s)^*+  S(t-s) K(s) ( S(t-s) F)^* \dd s \\ 
	&\quad + \int_{0}^{t} S(t-s) B ( S(t-s) B)^* \dd s,
	\end{split}
	\end{equation}
	taking values in the space $\cL_1(H)$, with the adjoint being taken with respect to the inner product of $H$. Here $S = ( S(t))_{t \in [0,T]}$ is a family of $\cL(\hat{H},H)$-valued operators, not yet assumed to be a semigroup, while $F \in \cL(H,\hat{H})$, $B \in \cL_2(U_0,\hat{H})$ and $Q_\xi \in \cL_1(H) \cap \Sigma^+(H)$. We assume that for any $v \in H$, the mapping $t \mapsto  S(t) v \in H$ is continuous on $[0,T]$ and the mappings  $t \mapsto  S(t) F v \in H$ and $t \mapsto  S(t) B \in \cL_2(U_0,H)$ are continuous on $(0,T]$. Below, we make an assumption on the boundedness of these mappings, which is used to deduce existence and uniqueness of solutions to~\eqref{eq:C} and the stochastic evolution equation
	\begin{equation}
	\label{eq:X}
	X(t) = {S}(t) X(0) + \int_{0}^{t} {S}(t-s) F X(s) \dd s + \int_0^t {S}(t-s) B \dd W(s), \text{ for } t \in (0,T].
	\end{equation}
	When $S$ is a semigroup, this is the mild solution of~\eqref{eq:introSPDE}. Here $X(0) = \xi$ is a Gaussian (possibly deterministic) $\cF_0$-measurable $H$-valued random variable.
	The stochastic integral is of the It\^o kind \cite[Chapter 4]{DPZ14}.
	
	\begin{assumption} 
		\label{assumptions:1}
		There is a constant $C<\infty$ and functions $a \in L^1([0,T],\R), b \in L^2([0,T],\R)$ such that $\norm{ S (t)}{\cL(H)} \le C$ for all $t \in [0,T]$ and $\norm{ S (t) F}{\cL(H)} \le  a(t)$, $\norm{ S (t) B}{\cL_2(U_0,H)} \le b(t)$ for all $t\in (0,T]$.
	\end{assumption}
	
	We look for a solution $K$ to~\eqref{eq:C} in the space $\cC([0,T],\cL_1(H))$ of continuous mappings with values in $\cL_1(H)$. This is a Banach space with norm $\norm{f}{\infty,\cL_1(H)} = \sup_{t \in [0,T]} \norm{f(t)}{\cL_1(H)}$.
	
	\begin{proposition}
		\label{prop:Kexists}
		Under Assumption~\ref{assumptions:1}, there is a unique solution $K \in \cC([0,T], \cL_1(H))$ to~\eqref{eq:C} such that $K(t) \in \Sigma(H)$ for all $t \in [0,T]$.
	\end{proposition}
	
	\begin{proof}
		First we note that for $v \in H$ fixed, we may write
		\begin{equation}
		\label{eq:prop:Kexists:1}
		\begin{split}
		S(t) Q_\xi S(t)^* v - S(s) Q_\xi S(s)^* v &%
		= \sum_{j = 1}^\infty (\inpro[H]{v}{(S(t) - S(s)) e_j}) S(t) \mu_j e_j \\
		&\quad+ \sum_{j = 1}^\infty \inpro[H]{v}{S(s) e_j} (S(s)-S(t)) \mu_j e_j,
		\end{split}
		\end{equation}
		where $(e_j)_{j=1}^\infty$ is an orthonormal eigenbasis of $Q_\xi$ with corresponding eigenvalues $(\mu_j)_{j=1}^\infty$. Since $(\mu_j)_{j=1}^\infty \in \ell^1$, both sums are well-defined operators in $\cL(H)$ applied to $v$. By~\eqref{eq:trace_norm} we get
		\begin{align*}
		\norm{S(t) Q_\xi S(t)^* - S(s) Q_\xi S(s)^*}{\cL_1(H)} &\le \sum_{j = 1}^\infty \norm{(S(t)-S(s))e_j}{H} \norm{\mu_j S(t) e_j}{H} \\ 
		&\quad+ \sum_{j = 1}^\infty \norm{(S(t)-S(s)) \mu_j e_j}{H} \norm{S(s) e_j}{H}.
		\end{align*} 
		Hence $S(\cdot)Q_\xi S(\cdot)^* \in \cC([0,T], \cL_1(H))$ as a consequence of Assumption~\ref{assumptions:1} and the dominated convergence theorem.
		By~\eqref{eq:schatten_bound_2}, the mapping $[0,t) \ni s \mapsto S(t-s) B ( S(t-s) B)^*$ takes values in the separable Banach space $\cL_1(H)$ and it is continuous so that the Bochner integral of it is well-defined. Similarly, for $\cK \in \cC([0,T], \cL_1(H))$, the mapping $[0,t) \ni s \mapsto S(t-s) F \cK(s)  S(t-s)^*+  S(t-s) \cK(s) ( S(t-s) F)^*$ takes values in $\cL_1(H)$ and it can be seen to be continuous by applying~\eqref{eq:trace_characterization} along with a calculation similar to~\eqref{eq:prop:Kexists:1}. The mapping 
		\begin{equation}
		\label{eq:prop:Kexists:2}
		\begin{split}
		\cK \mapsto& S(\cdot) Q_\xi S(\cdot)^* + \int_{0}^{\cdot}  S(\cdot-s) F \cK(s)  S(\cdot-s)^*+  S(\cdot-s) \cK(s) ( S(\cdot-s) F)^* \dd s \\ 
		&\quad + \int_{0}^{\cdot} S(\cdot-s) B ( S(\cdot-s) B)^* \dd s 
		\end{split}
		\end{equation}
		from $\cC([0,T], \cL_1(H))$ into itself is therefore well-defined. By Assumption~\ref{assumptions:1}, 
		\begin{align*}
		\Big\| &\int_{0}^{t}  S(t-s) F \cK_1(s)  S(t-s)^*+  S(t-s) \cK_1(s) ( S(t-s) F)^* \dd s \\ &\quad-\int_{0}^{t}  S(t-s) F \cK_2(s)  S(t-s)^*+  S(t-s) \cK_2(s) ( S(t-s) F)^* \dd s \Big\|_{\cL_1(H)} \\
		&\lesssim e^{\sigma t} \int_{0}^{t} a(t-s) e^{-\sigma (t-s)} \left(e^{-\sigma s} \norm{\cK_1(s) - \cK_2(s)}{\cL_1(H)}\right) \dd s
		\end{align*} 
		for arbitrary $\sigma \in \R$ and $\cK_1, \cK_2 \in \cC([0,T],\cL_1(H))$. Therefore, since $\lim_{\sigma \to \infty} \int^T_0 a(s) e^{-\sigma s} \dd s = 0$ by the dominated convergence theorem, the mapping~\eqref{eq:prop:Kexists:2} is a contraction with respect to the norm defined by $\sup_{t \in [0,T]} e^{-\sigma t} \norm{\cK(t)}{\cL_1(H)}$ for sufficiently large $\sigma \ge 0$. 
		This norm is equivalent to $\norm{\cdot}{\infty,\cL_1(H)}$, so
		the Banach fixed point theorem yields existence and uniqueness of $K$ as the limit of the sequence $(K_n)_{n=0}^\infty \subset \cC([0,T], \cL_1(H))$. Here $K_0 = 0$ and $K_n$, $n \ge 1$, is given by 
		\begin{align*}
		K_{n}(t) &= S(t) Q_\xi S(t)^* + \int_{0}^{t}  S(t-s) F K_{n-1}(s)  S(t-s)^*+  S(t-s) K_{n-1}(s) ( S(t-s) F)^* \dd s \\ 
		&\quad + \int_{0}^{t} S(t-s) B ( S(t-s) B)^* \dd s, t \in [0,T].
		\end{align*}
		Clearly $K_0(t) \in \Sigma(H)$ for all $t \in [0,T]$. By induction,
		$K_n(t) \in \Sigma(H)$ for all $n \in \N$, $t \in [0,T]$. Since convergence of $(K_n)_{n=0}^\infty$ in $\cC([0,T], \cL_1(H))$ yields convergence of $(K_n(t))_{n=0}^\infty$ in $\cL(H)$ we have  $\inpro[H]{K(t) u}{v} = \lim_n \inpro[H]{K_n(t) u}{v}$ for all $t \in [0,T]$ and $u,v \in H$. Therefore, $K(t) \in \Sigma(H)$ for all $t \in [0,T]$.
	\end{proof}
	
	\begin{remark}
		Since the mapping $t \mapsto  S(t) F v \in H$ is allowed to be discontinuous at $0$, the advection term in~\eqref{eq:advdiff-intro} could be treated as a linear perturbation $F$, at least in the case of Dirichlet boundary conditions, see \cite[Example~2.22]{K14}. In this case $S$ is the semigroup generated by the elliptic operator of~\eqref{eq:advdiff-intro}. For notational convenience and to easily treat more general boundary conditions, we instead choose to, in Section~\ref{sec:heat}, treat the advection term in~\eqref{eq:advdiff-intro} as part of an elliptic operator.
	\end{remark}
	
	The next proposition confirms that the solution $X$ to the stochastic evolution equation~\eqref{eq:X} is Gaussian at all times $t \in [0,T]$. As a consequence, $K(t)$ therefore determines the distribution of $X(t)$ when $S$ is a semigroup, since by Theorem~\ref{thm:covarianceisK} below, $K(t)=\Cov(X(t))$ for $t \in[0,T]$.
	
	\begin{proposition}
		\label{prop:Xexists}
		Under Assumption~\ref{assumptions:1}, there is a unique solution $X \in \cC([0,T], L^2(\Omega,H))$ to~\eqref{eq:X} and $X(t)$ is Gaussian for all $t \in [0,T]$.
	\end{proposition}
	\begin{proof}
		In the case that $S$ is a semigroup, $\xi$ is deterministic and $F=0$, the result is well-known, see, e.g., \cite[Theorem~5.2]{DPZ14}. We only sketch the proof in our general case. Existence and uniqueness of a solution to~\eqref{eq:X} follow from a Banach fixed point theorem as in Proposition~\ref{prop:Kexists}, using the It\^o isometry for the stochastic integral \cite[Theorem~2.25]{K14}. In particular, we have existence and uniqueness of the process
		\begin{equation*}
		X_1 = {S}(\cdot) \xi + \int_{0}^{\cdot} {S}(\cdot-s) F X_1(s) \dd s
		\end{equation*}
		in $\cC([0,T], L^2(\Omega,H))$ and of the process
		\begin{equation*}
		X_2 = \int_{0}^{\cdot} {S}(\cdot-s) F X_2(s) \dd s + \int_0^\cdot {S}(\cdot-s) B \dd W(s)
		\end{equation*}
		in $\cC([0,T], L^2(\Omega,H))$ as limits of iterative sequences $(X^n_1)_{n=0}^\infty$ and $(X^n_2)_{n=0}^\infty$ as in the proof of Proposition~\ref{prop:Kexists}, with $X^0_1 = X^0_2 = 0$. Since $X_1^n(t)$ is obtained from a linear and bounded transformation of $\xi$, it is Gaussian for each $t \in [0,T]$. For $X_2^n$, one can use an inductive argument along with the stochastic Fubini theorem \cite[Theorem~4.18]{K14} to see that there is a function $\psi_n \colon [0,T] \times [0,T] \to \cL_2(U_0,H)$, continuous in each argument, with $\sup_{t \in [0,T]} \norm{\psi_n(t,\cdot)}{L^2([0,t],\cL_2(U_0,H))} < \infty$ such that $X_2^n(t) = \int^t_0 \psi_n(t,s) \dd W(s)$ for all $t \in [0,T]$. Therefore $X_2^n(t)$ is also Gaussian for each $t \in [0,T]$. Since $\lim_n X_1^n(t) \oplus X^{n}_2(t) = X_1(t) \oplus X_2(t)$ in $L^2(\Omega,H \oplus H)$, $(X_1, X_2)$ is a jointly Gaussian pair, from which the result follows. 
	\end{proof}
	
	We now prove our main result, connecting the equations~\eqref{eq:C} and~\eqref{eq:X}. For this we need to assume that $S$ is a $C_0$-semigroup. This is required by the main tool of our proof, the mild It\^o formula \cite[Theorem~1]{DPJR19}. In our setting, this formula gives that for a twice continuously Fr\'echet differentiable functional $\varphi$ and $t \in [0,T]$,
	\begin{align*}
	\notag \varphi(X(t)) &= \varphi(S(t)\xi) + \int^t_0 \inpro[H]{\varphi'(S(t-s)X(s))}{S(t-s) FX(s)} \dd s \\ &\qquad+ \int^t_0 \inpro[H]{\varphi'(S(t-s)X(s))}{S(t-s) B \cdot} \dd W(s) \\ &\qquad+ \frac{1}{2} \sum_{j = 1}^\infty \int^t_0 \inpro[H]{\varphi''(S(t-s)X(s)) S(t-s) B f_j}{S(t-s) B f_j}.
	\end{align*}
	Here the Fr\'echet derivatives $\varphi'$ and $\varphi''$ take, by the Riesz representation theorem, values in $H$ and $\Sigma(H)$, respectively, and $(f_j)_{j= 1}^\infty$ is an orthonormal basis of $U_0$. The formula is proven by applying the standard It\^o formula to~\eqref{eq:X} with $t$ fixed in the integrand, and then using the semigroup property of $S$ to relate the process with $t$ fixed to the original process $X$. 
	
	\begin{theorem}
		\label{thm:covarianceisK}
		Let $\Cov(\xi) = Q_\xi$ and let $S$ be a $C_0$-semigroup satisfying Assumption~\ref{assumptions:1}. Then, the process $K = (\Cov(X(t)))_{t \in [0,T]}$, where $X$ is given by~\eqref{eq:X}, is the solution of~\eqref{eq:C}.
	\end{theorem}
	\begin{proof}
		We first suppose that $X(0) = \xi = 0$ so that $\Cov(X(t)) = \E[X(t) \otimes X(t)]$. By an argument analogous to~\eqref{eq:introcovboundedbystrong}, the continuity of $X$ implies that $\Cov(X(\cdot)) \in \cC([0,T],\cL_1(H))$. Below, we will make several interchanges of integration, summation and expectation. These are allowed by Fubini's theorem, using Assumption~\ref{assumptions:1} and the fact that~$\sup_{t \in [0,T]}\norm{X(t)}{L^2(\Omega,H)} < \infty$. Let $(e_j)_{j = 1}^\infty$ and $(f_j)_{j= 1}^\infty$ be orthonormal bases of $H$ and $U_0$, respectively. 
		By~Assumption~\ref{assumptions:1}, the mild It\^o formula is applicable with the functional $\varphi$ given by $\varphi(x) = \inpro[H]{x}{e_i}\inpro[H]{x}{e_j}$ for $x \in H$, $i,j \in \N$. We have $\varphi'(x) = \inpro[H]{x}{e_i}e_j + \inpro[H]{x}{e_j}e_i$ and $\varphi''(x) = e_i \otimes e_j + e_j \otimes e_i$. Using also the zero expectation property of the It\^o integral and~\eqref{eq:HSinpro}, we find that 
		\begingroup
		\allowdisplaybreaks
		\begin{align*}
		\inpro[\cL_2]{\Cov(X(t))}{e_i \otimes e_j} &= \E[\inpro[H]{X(t)}{e_i}\inpro[H]{X(t)}{e_j}] \\ 
		&= \int^t_0 \E[\inpro[H]{S(t-s)FX(s)}{e_i} \inpro[H]{S(t-s)X(s)}{e_j}] \dd s \\
		&\quad+ \int^t_0 \E[\inpro[H]{S(t-s)X(s)}{e_i} \inpro[H]{S(t-s)FX(s)}{e_j}] \dd s\\
		&\quad+\sum_{n=1}^\infty \int^t_0 \inpro[H]{S(t-s)B f_n}{e_i} \inpro[H]{S(t-s)B f_n}{e_j} \dd s,
		\end{align*}
		for $t \in [0,T], i,j \in \N$. By~\eqref{eq:HSinpro},~\eqref{eq:HSlinops} and the definition of~$\inpro[U_0]{\cdot}{\cdot}$, this is equal to
		\begin{align*} 
		&\int^t_0 \inpro[\cL_2(H)]{\E[S(t-s)F X(s) \otimes S(t-s)X(s)]}{e_i \otimes e_j} \dd s \\ 
		&\quad+\int^t_0 \inpro[\cL_2(H)]{\E[S(t-s) X(s) \otimes S(t-s)F X(s)]}{e_i \otimes e_j} \dd s \\
		&\quad+ \int^t_0 \inpro[U_0]{(S(t-s)B)^*e_i}{(S(t-s)B)^*e_j} \dd s \\
		&= \int^t_0 \inpro[\cL_2(H)]{S(t-s)F\Cov(X(s))S(t-s)^*}{e_i \otimes e_j} \dd s \\
		&\quad+ \int^t_0 \inpro[\cL_2(H)]{S(t-s)\Cov(X(s))(S(t-s)F)^*}{e_i \otimes e_j} \dd s \\
		&\quad+ \int^t_0 \inpro[\cL_2(H)]{S(t-s)B(S(t-s)B)^*}{e_i \otimes e_j} \dd s.
		\end{align*}%
		\endgroup
		By applying $\inpro[H]{\cdot}{e_i\otimes e_j}$ to~\eqref{eq:C} and using~\eqref{eq:HSinpro}, we obtain similarly that
		\begin{align*}
		\inpro[\cL_2(H)]{K(t)}{e_i \otimes e_j} &= \int^t_0 \inpro[\cL_2(H)]{S(t-s)FK(s)S(t-s)^*}{e_i \otimes e_j} \dd s \\
		&\quad+ \int^t_0 \inpro[\cL_2(H)]{S(t-s)K(s)(S(t-s)F)^*}{e_i \otimes e_j} \dd s \\
		&\quad+ \int^t_0 \inpro[\cL_2(H)]{S(t-s)B(S(t-s)B)^*}{e_i \otimes e_j} \dd s.
		\end{align*}
		Combining this with the previous result yields
		\begin{align*}
		&\inpro[\cL_2(H)]{K(t)-\Cov(X(t))}{e_i \otimes e_j} \\
		&\quad= \int^t_0 \inpro[\cL_2(H)]{S(t-s)F\left(K(s)-\Cov(X(s))\right)S(t-s)^*}{e_i \otimes e_j} \dd s \\
		&\quad\quad+ \int^t_0 \inpro[\cL_2(H)]{S(t-s)\left(K(s)-\Cov(X(s))\right)(S(t-s)F)^*}{e_i \otimes e_j} \dd s.
		\end{align*}
		This shows that 
		\begin{align*}
		K(t) - \Cov(X(t)) &= \sum_{i,j = 1}^{\infty} \inpro[\cL_2(H)]{K(t)-\Cov(X(t))}{e_i \otimes e_j} e_i \otimes e_j \\
		&=\int_{0}^{t}  S(t-s) F \left(K(s)-\Cov(X(s))\right)  S(t-s)^* \\
		&\quad \qquad+  S(t-s) \left(K(s)-\Cov(X(s))\right) ( S(t-s) F)^* \dd s
		\end{align*}
		so that $K(t)  = \Cov(X(t))$ for all $t \in [0,T]$ by uniqueness of $K$ in~\eqref{eq:C}. 
		
		For the general case $\xi \neq 0$, we write $X = X_1 + X_2$ as in the proof of Proposition~\ref{prop:Xexists}. From the fact that $X_1(t) \oplus X_2(t)=\lim_n X_1^n(t) \oplus X^{n}_2(t)$ in $L^2(\Omega,H \oplus H)$, we find that $\Cov(\inpro[H]{X_1(t)}{u},\inpro[H]{X_2(t)}{v}) = \lim_n \Cov(\inpro[H]{X^n_1(t)}{u},\inpro[H]{X^n_2(t)}{v}) = 0$.
		This implies that $\Cov(X_1(t),X_2(t)) = \Cov(X_2(t),X_1(t)) = 0$ for arbitrary $t \in [0,T]$. Since $X_2(0) = 0$, 
		\begin{align*}
		\Cov(X_2(t)) &= \int_{0}^{t}  S(t-s) F \Cov(X_2(s))  S(t-s)^*+  S(t-s) \Cov(X_2(s)) ( S(t-s) F)^* \dd s \\ 
		&\quad + \int_{0}^{t} S(t-s) B ( S(t-s) B)^* \dd s
		\end{align*}
		as a consequence of what we have already shown. A similar argument using the mild It\^o formula yields
		\begin{align*}
		\Cov(X_1(t)) &= S(t) Q_\xi S(t)^* \\
		&\quad+  \int_{0}^{t}  S(t-s) F \Cov(X_1(s))  S(t-s)^*+  S(t-s) \Cov(X_1(s)) ( S(t-s) F)^* \dd s
		\end{align*}
		for all $t \in [0,T]$. The proof is completed by noting that 
		\begin{equation*}
		\Cov(X_1(t) + X_2(t)) = \Cov(X_1(t)) + \Cov(X_1(t),X_2(t)) + \Cov(X_2(t),X_1(t)) + \Cov(X_2(t)). \qedhere
		\end{equation*}
	\end{proof}
	
	\begin{remark}
		As a consequence of the theorem, $K(t) \in \Sigma^+(H)$ for all $t \in [0,T]$. %
	\end{remark}
	
	We finish this section with a general error decomposition formula with respect to the integral equation~\eqref{eq:C} and an approximation of the semigroup $S$. For this we consider a family $\tilde S = (\tilde S(t))_{t \in [0,T]}$ of operators in $\cL(\hat{H},H)$ such that the mappings $t \mapsto  \tilde S(t) v \in H$,  $t \mapsto  \tilde S(t) F v \in H$ and $t \mapsto  \tilde S(t) B \in \cL_2(U_0,H)$ are continuous almost everywhere on $[0,T]$ for all $v \in H$. If~$\tilde S$ also satisfies~Assumption~\ref{assumptions:1} and we consider a function $\hat K : [0,T] \to \cL_1(H) \cap \Sigma(H))$ that is continuous almost everywhere on~$[0,T]$, then $\tilde K(t)$, given by 
	\begin{equation}
	\label{eq:tilde-K}
	\begin{split}
	\tilde{K}(t) &= \tilde S(t) Q_\xi \tilde S(t)^* + \int_{0}^{t}  \tilde S(t-s) F \hat K(s)  \tilde S(t-s)^*+  \tilde S(t-s) \hat K(s) ( \tilde  S(t-s) F)^* \dd s \\ 
	&\quad + \int_{0}^{t} \tilde S(t-s) B ( \tilde S(t-s) B)^* \dd s,
	\end{split}
	\end{equation}	
	is well-defined since the integrands are measurable mappings with values in the separable Banach space $\cL_1(H)$. In the error decomposition formula we consider $\tilde K(t)$, defined by~\eqref{eq:tilde-K}, as an approximation of $K(t)$, $t \in [0,T]$.
	
	\begin{proposition}
		\label{prop:Kerr}
		Let, for $t \in [0,T]$, $K(t)$ be given by~\eqref{eq:C} and $\tilde K(t)$ by~\eqref{eq:tilde-K}. Then, with $\cO^+(s) = S(s) + \tilde{S}(s)$ and $\cO^-(s) = S(s) - \tilde{S}(s)$ for $s\in[0,t]$,
		\begingroup
		\allowdisplaybreaks
		\begin{align*}
		\norm{K(t)-\tilde K(t)}{\cL_i(H)} &\le \norm{\cO^-(t) Q_\xi \cO^+(t)}{\cL_i(H)} \\
		&\quad+ 2 \int^t_0 \norm{S(t-s) F (K(s)-\hat K(s))S(t-s)^*}{\cL_i(H)} \dd s \\
		&\quad+ \int^t_0 \norm{\cO^-(t-s) F \hat K(s)(\cO^+(t-s))^*}{\cL_i(H)} \dd s \\ 
		&\quad+ \int^t_0 \norm{\cO^+(t-s) F \hat K(s)(\cO^-(t-s))^*}{\cL_i(H)} \dd s \\
		&\quad+ \int^t_0 \norm{\cO^-(t-s) B (\cO^+(t-s)B)^*}{\cL_i(H)} \dd s \text{ for } i \in \{1,2\}.
		\end{align*}%
		\endgroup
	\end{proposition}
	\begin{proof}
		The proposition is a straightforward consequence of the triangle inequality,~\eqref{eq:schattenadjoint}, the fact that $K(t), \tilde K(t) \in \Sigma(H)$ for $t \in [0,T]$, and the identity
		\begin{equation}
		\label{eq:operator-identity}
		\Gamma_1 \tilde \Gamma \Gamma_1^* - \Gamma_2 \tilde \Gamma \Gamma_2^* = \frac{1}{2} \left( (\Gamma_1 + \Gamma_2) \tilde \Gamma (\Gamma_1 - \Gamma_2)^* + (\Gamma_1 - \Gamma_2) \tilde \Gamma (\Gamma_1 + \Gamma_2)^* \right)
		\end{equation}
		for $\Gamma_1, \Gamma_2, \tilde\Gamma \in \cL(H)$.
	\end{proof}
	
	\section{Applications}
	\label{sec:applications}
	
	We apply the theory of the previous section to two concrete stochastic equations, a stochastic advection-diffusion equation and the stochastic wave equation. Fully discrete approximation schemes are analysed and numerical simulations are provided for illustration.
	
	\subsection{A stochastic advection--diffusion equation}
	\label{sec:heat}
	Let $\cD \subset \R^d$, $d = 1, 2, 3$ be a bounded domain. The SPDE we consider in this section is formally given by
	\begin{equation}
	\label{eq:advdiff}
	\begin{alignedat}{2}		
	\dd X(t,x) + \cA X(t,x) \dd t&= \dd W(t,x) &&\text{ for } t \in (0,T], x \in \cD, \\
	X(0,x) &= \xi(x),   &&\text{ for } x \in \cD,
	\end{alignedat}
	\end{equation}
	for a random initial condition $X(0) = \xi$ and an operator
	\begin{equation*}
	\cA = -\sum^{d}_{i,j=1} \frac{\partial}{\partial x_i} a_{i,j}  \frac{\partial}{\partial x_j}  + \sum_{j = 1}^d a_j \frac{\partial}{\partial x_j} + a_0.
	\end{equation*}
	We consider either Dirichlet, Neumann or Robin boundary conditions and we let $W$ be a generalized Wiener process in $H = L^2(\cD)$ with covariance operator $Q \in \Sigma^+(H)$.  We assume that $\xi$ is an $H$-valued $\cF_0$-measurable Gaussian random variable with covariance $Q_\xi$. The coefficients $a_{i,j}, a_{i}$, $i,j = 1, \ldots, d$, and $a_0$ are functions on $\bar{\cD}$ fulfilling $a_{i,j} = a_{j,i}$. We assume that for there is some $\lambda_0 > 0$ such that $\sum^{d}_{i,j=1} a_{i,j} (x) y_i y_j \ge \lambda_0 |y|^2$ for all $y \in \R^d$ and $x \in \bar{\cD}$, so that $\cA$ is elliptic.

	To put~\eqref{eq:advdiff} into our framework, we follow~\cite{FS91} and introduce the spaces $V$ and $\bH$ as subspaces of the Sobolev spaces $H^1 = H^1(\cD)$ and $H^2 = H^2(\cD)$, respectively. In the Dirichlet case we set $V = \bH = H^1_0 = \{u \in H^1 : u = 0 \text{ on } \partial \cD \}$.
	In the Robin case (Neumann boundary conditions being a special case thereof), we set $V =H^1(\cD)$ and $\bH = \{u \in H^2 : {\partial u}/{\partial \nu_\Lambda} + \sigma u = 0 \text{ on } \partial \cD \},$
	where $\sigma \colon \partial \cD \to \R$ is a sufficiently smooth function and  
	\begin{equation*}
	\frac{\partial u}{\partial \nu_\Lambda} = \sum_{i,j = 1}^d n_i a_{i,j} \frac{\partial u}{\partial x_j},
	\end{equation*}
	with $\mathbf{n} = (n_1, \ldots, n_d)$ being the outward unit normal to $\partial \cD$. Here and below, the boundary conditions should be understood in terms of trace operators, see~\cite[Sections 1.5-1.6]{G85}. We define a bilinear form $\lambda \colon V \times V \to \R$ associated with $\cA$ by
	\begin{equation*}
	\lambda(u,v) = \int_\cD  \sum^{d}_{i,j=1}  a_{i,j} \frac{\partial u}{\partial x_i} \frac{\partial v}{\partial x_j} + \sum_{j = 1}^d a_j \frac{\partial u}{\partial x_j} v  + a_0 u v \dd x + \int_{\partial \cD} \sigma u v \dd x,
	\end{equation*}
	where the last term is dropped in the Dirichlet case. If the coefficients are bounded, then $|\lambda(u,v)| \le \norm{u}{V} \norm{v}{V}$ so that we may associate an operator $\Lambda \colon V \to V^*$ to $\lambda$ by $\lambda(u,v) = {}_{V^*} \langle \Lambda u, v \rangle_V$. By Riesz's representation theorem, we obtain a Gelfand triple $V \subset H \subset V^*$. We restrict $\Lambda$ to $\dom(\Lambda) = \{u \in V : \Lambda u \in H\}$ without changing notation. 
	When $\cD$ has Lipschitz boundary, one may show (using the trace inequality~\cite[Theorem~1.5.1.3]{G85} if necessary) that there are two constants $\tilde \lambda_0 > 0$, $c_0 \ge 0$ such that $\lambda(u,u) \ge \tilde \lambda_0 \norm{u}{H^1}^2 - c_0 \norm{u}{H}^2$ for $u \in V$.
	We add the term $c_0 X(t,x)$ to both sides of~\eqref{eq:advdiff}. Then the associated bilinear form $a(\cdot,\cdot) = \lambda(\cdot,\cdot) + c_0 \inpro[H]{\cdot}{\cdot}$ is coercive. 
	In the case that $\sigma \ge 0$, the constant $c_0$ may, for example, be chosen as any number fulfilling
	\begin{equation}
	\label{eq:advdiff:c0}
	c_0 > \frac{\sup_{x \in \cD} \sum_{j = 1}^{d} |a_j(x)|}{4 \lambda_0 \epsilon} - \inf_{x \in \cD} a_0(x),
	\end{equation}
	for an arbitrary $\epsilon \in (0,1)$. In the Dirichlet case, we may pick $c_0 = 0$ if $a_j = 0$ for all $j = 1, \ldots, d$ and $a_0(x) \ge 0$. With this, the SPDE~\eqref{eq:advdiff} is put into the form of~\eqref{eq:introSPDE} by letting $A = \Lambda + c_0 I$ and $F = c_0 I$. 
	The adjoint operator $A^*$ is defined by associating it with the bilinear form $a^*$, given by $a^*(u,v) = a(v,u)$ for $u,v \in V$. For smooth coefficients, we use Green's formula to see that $A^*$ can be identified with the formal adjoint of $\cA$ perturbed by $c_0$ (cf.\ \cite[Section 2.1.3]{Y10}). It is given by 
	\begin{equation*}
	\cA^* = -\sum^{d}_{i,j=1} \frac{\partial}{\partial x_i} a_{i,j}  \frac{\partial}{\partial x_j} - \sum_{j = 1}^d a_j \frac{\partial}{\partial x_j} - \sum_{j = 1}^d \frac{\partial a_j}{\partial x_j} + a_0 + c_0.
	\end{equation*}
	In the Robin case, the boundary conditions of $\cA^*$ change to the ones of the space
	\begin{equation*}
	\bH^* = \left\{ u \in H^2 : \frac{\partial u}{\partial \nu_\Lambda} + \left(\sum_{j=1}^{d} a_j n_j + \sigma\right) u = 0 \text{ on } \partial \cD  \right\}, 
	\end{equation*}
	while in the Dirichlet case, $\bH^* = \bH$.
	We also need the symmetrized operator $A_0$ associated with the bilinear form $a_0 = (a+a^*)/2$. Like $A^*$, $A_0$ is identified with a differential operator 
	\begin{equation*}
	\cA_0 = -\sum^{d}_{i,j=1} \frac{\partial}{\partial x_i} a_{i,j}  \frac{\partial}{\partial x_j} - \frac{1}{2} \sum_{j = 1}^d \frac{\partial a_j}{\partial x_j} + a_0 + c_0.
	\end{equation*}
	We set
	\begin{equation*}
	\bH_0 = \left\{ u \in H^2 : \frac{\partial u}{\partial \nu_\Lambda} + \left(\frac{1}{2} \sum_{j=1}^{d} a_j n_j + \sigma\right) u = 0 \text{ on } \partial \cD  \right\}, 
	\end{equation*}
	in the Robin case and $\bH_0 = \bH$ in the Dirichlet case.
	With these notions in place, we introduce an assumption of elliptic regularity.
	\begin{assumption}
		The coefficients $a_{i,j}, a_j, a_0, \sigma$, $i,j=1,\ldots,d$, are sufficiently smooth and $\cD$ is sufficiently regular, with $\partial \cD$ at least Lipschitz, to guarantee that $\dom(A) = H^2 \cap \bH$, $\dom(A^*) = H^2 \cap \bH^*$ and $\dom(A_0) = H^2 \cap \bH_0$. The equalities hold with equivalence of $\norm{\cdot}{H^2}$ and the graph norms $\norm{A\cdot}{H}, \norm{A^*\cdot}{H}$ and $\norm{A_0\cdot}{H}$, respectively.
	\end{assumption}
	We refer to~\cite{G85} for details on when this assumption holds. 
	It is satisfied when $\cD$ is convex and $a_{i,j}, a_j, a_0, \sigma$, $i,j=1,\ldots,d$ are infinitely differentiable, see \cite{FS91}.
	
	Since $a$ is coercive, $A$ is a sectorial operator. Negative fractional powers of $A$ are therefore well-defined as elements of $\cL(H)$ given by
	\begin{equation*}
	A^{-\frac{s}{2}} = \frac{1}{2 \pi i} \int_\gamma \lambda^{-\frac{s}{2}} (\lambda-A)^{-1} \dd \lambda,
	\end{equation*}
	where $\gamma$ is a counterclockwise oriented contour surrounding the spectrum of $A$. Positive fractional powers are densely defined closed operators on $H$ defined by $A^{{s/2}} = (A^{-{s/2}})^{-1}$ \cite[Section~2.1.7]{Y10}. We note that $(A^*)^{s/2} = (A^{s/2})^*$ for all $s \in \R$. Moreover, by \cite[Theorem~3.1]{K61} and \cite[Th\'eor\`eme~6.1]{L62} (applicable since $\cD$ has Lipschitz boundary) we have $\dom(A^{s/2}) = \dom((A^*)^{s/2}) = \dom(A_0^{s/2})$ for all $s \in [0,1]$ with norm equivalence. For $s=1$, these spaces can also be identified with $V$. 
	
	It is convenient to express our regularity assumptions on $Q$ not in terms of fractional powers of $A$, which is usually the case when $A$ is self-adjoint, but of $A_0$. As $A_0$ is positive definite with a compact inverse (a consequence of~\cite[Theorem~1.38]{Y10} since $\dom(A_0) \subset H^2$), its fractional powers can be characterized in a simple way by the spectral theorem, cf.\ \cite[Appendix~B.2]{K14}. For $s\ge 0$, we write $\dot{H}^s$ for the Hilbert space $\dom(A_0^{{s/2}})$. Moreover, Hilbert spaces $\dot{H}^{-s}$ are well-defined as completions of sequences in $H$ with respect to~$\norm{\cdot}{\dot{H}^{-s}} = \norm{A_0^{-{s/2}}\cdot}{H}$. In this way, we obtain a set~$(\dot{H}^{s})_{s \in \R}$ of Hilbert spaces, with $\dot{H}^\alpha \hookrightarrow \dot{H}^s$ for $\alpha > s$, continuously and densely. Lemma~2.1 in~\cite{BKK20} allows us to, for all $\alpha, s \in \R$, extend $A_0^{{s/2}}$ to an operator in $\cL(\dot{H}^\alpha,\dot{H}^{\alpha-s})$, and we do so without changing notation. 
	Note, that for $v \in H$ and $s \in [0,1]$,
	\begin{align*}
	\norm{A_0^{-\frac{s}{2}} v}{H} = \sup_{\substack{w \in H \\ \norm{w}{H}=1}} \left|\inpro[H]{A_0^{-\frac{s}{2}} v}{w}\right| &= \sup_{\substack{w \in H \\ \norm{w}{H}=1}} \left|\inpro[H]{A^{-\frac{s}{2}}v}{(A^\frac{s}{2})^* A_0^{-\frac{s}{2}} w}\right| \\
	&\le \norm{A^{-\frac{s}{2}} v}{H} \sup_{\substack{w \in H \\ \norm{w}{H}=1}} \norm{(A^\frac{s}{2})^* A_0^{-\frac{s}{2}} w}{H} \lesssim \norm{A^{-\frac{s}{2}} v}{H},
	\end{align*}
	by the equivalence $\dom(A_0^{{s/2}}) = \dom((A^*)^{{s/2}})$. Similarly,
	\begin{equation}
	\label{eq:advdiff:symnormequiv2}
	\norm{A^{-\frac{s}{2}} v}{H} \lesssim \norm{A_0^{-\frac{s}{2}} v}{H}.
	\end{equation}
	This is true also for $A^*$. Therefore, $A^{{s/2}}$ and $(A^*)^{{s/2}}$ can be considered as operators in $\cL(\dot{H}^s,H)$ for $s \in [-1,1]$.
	
	The operator $-A$ is the infinitesimal generator of a uniformly bounded analytic semigroup $S$ on $H$ (see \cite{FS91}), which yields a mild solution~\eqref{eq:X} to the SPDE~\eqref{eq:advdiff}. The semigroup maps into $\dom(A^s)$ for all $s \ge 0$, with $S$ and $A^{s}$ commutative on $\dom(A^s)$. The stability estimate
	\begin{equation}
	\label{eq:advdiff:semistab}
	\norm{A^{s}S(t)}{\cL(H)} \lesssim t^{-s}
	\end{equation}
	is satisfied for $t > 0$. Moreover, as seen in \cite[Section~2.7.7]{Y10},
	\begin{align}
	\label{eq:advdiff:semiholder}
	\norm{A^{-s}(S(t)-I)}{\cL(H)} &\lesssim t^{s},  s \in [0,1], t \ge 0.
	\end{align}
	
	We now make the following assumption on $Q$.
	\begin{assumption}
		\label{ass:advdiff:ass1}
		There is a constant $r \in (0,1]$ such that
		\begin{equation*} 
		\norm{A_0^{\frac{r-1}{2}} Q^{\frac{1}{2}}}{\cL_2(H)} = \norm{Q^{\frac{1}{2}}}{\cL_2(H,\dot{H}^{r-1})} < \infty.
		\end{equation*}
	\end{assumption}  
	With this assumption in place, we have for $t>0$ that
	\begin{align*}
	\norm{S(t)}{\cL_2(U_0,H)} &= \norm{A^{\frac{1-r}{2}} S(t) A^{\frac{r-1}{2}} Q^\frac{1}{2}}{\cL_2(H)} \le \norm{A^{\frac{1-r}{2}} S(t)}{\cL(H)} \norm{A^{\frac{r-1}{2}} Q^\frac{1}{2}}{\cL_2(H)} \lesssim 
	t^{\frac{r-1}{2}}.
	\end{align*}
	The last inequality is a consequence of~\eqref{eq:advdiff:symnormequiv2} and~\eqref{eq:advdiff:semistab}.
	To have Assumption~\ref{assumptions:1} fulfilled, we set $B=I$ and $F=c_0 I$. Here, the inclusion $B$ is regarded as an operator in $U_0=Q^{1/2}(H)$ to $\hat{H} = \dot{H}^{r-1}$, $F$ as an operator from $H$ to $\dot{H}^{r-1}$ and $S(t)$ as an operator from $\dot{H}^{r-1}$ to $H$. 
	Continuity of $(0,T] \ni t \mapsto  S(t) B \in \cL_2(U_0,H)$ is a consequence of the fact $\norm{S(t)}{\cL_2(U_0,H)}<\infty$ for all $t \in (0,T]$ along with continuity of the mapping $[0,T] \ni t \mapsto S(t)v$ for all $v \in U_0 = Q^{1/2}(H) \subset H$. With this Assumption~\ref{assumptions:1} is fulfilled, so that by Proposition~\ref{prop:Xexists} and the fact that $S$ is a semigroup, a predictable mild solution to~\eqref{eq:introSPDE} exists. Proposition~\ref{prop:Kexists} and Theorem~\ref{thm:covarianceisK} yield a unique solution $K$ to~\eqref{eq:C} such that $K(t) = \Cov(X(t))$, $t \in [0,T]$.
	
	We now move on to approximation of~\eqref{eq:C}. For this, we consider the same approximation of $S$ as in~\cite{FS91} and assume from here on that $\cD$ is a convex polygon. For the spatial discretization, we let $(V_h)_{h \in (0,1]} \subset {H}^1$ be a standard family of finite element spaces consisting of piecewise linear polynomials with respect to a regular family of triangulations of $\cD$ with maximal mesh size $h$, vanishing on $\partial \cD$ in the Dirichlet case. We assume the mesh to be quasi-uniform. The spaces are equipped with $\inpro[V_h]{\cdot}{\cdot} = \inpro[H]{\cdot}{\cdot}$. On this space, let $A_h\colon V_h \to V_h$ be given by $\inpro[H]{A_h v_h}{u_h} = a(v_h,u_h)$
	for all $v_h, u_h \in V_h$. Since $A_h$ is sectorial, fractional powers of it are defined in the same way as for $A$. We write $A_{0,h}$ for the operator defined in the same way using the bilinear form $a_0$. Note that $A_h^*$ coincides with the operator defined using $a^*$. By $P_h \colon \dot{H}^{-1} \to V_h$ we denote the generalized orthogonal projector defined by $\inpro[H]{P_h x}{y_h} = \langle A_0^{-1/2} x, A_0^{1/2} y_h \rangle_H$ for $x \in \dot{H}^{-1}, y_h \in V_h$. Since $a_0$ is a symmetric form, we have $\norm{A_{0,h}^{1/2} P_h v}{H}=\norm{A_0^{1/2} P_h v}{H} \lesssim \norm{P_h v}{V}$. The mesh of $V_h$ is assumed to be quasi-uniform, so $\sup_{h \in (0,1]} \norm{P_h}{\cL(V)} < \infty$ \cite[Proposition~3.2]{FS91}, \cite[Example~3.6]{K14}. The interpolation arguments of~\cite{AL16} then imply that for $s \in [-1,1]$ there is a constant $C<\infty$ such that for all $h \in (0,1]$,
	\begin{equation}
	\label{eq:advdiff:eqnorms2}
	\norm{A_{0,h}^{\frac{s}{2}} P_h A_0^{-\frac{s}{2}}}{\cL(H)} \le C.
	\end{equation}
	Moreover, by \cite[Theorem~3.1]{K61} (see also the proof of \cite[Theorem~5.3]{FS91}), for all $s \in [0,1)$ there is a constant $C<\infty$ such that for all $h \in (0,1]$, $\norm{A_{h}^{\frac{s}{2}} A_{0,h}^{-\frac{s}{2}} P_h}{\cL(H)} \le C$
	and	 
	\begin{equation}
	\label{eq:advdiff:eqnorms3}
	\norm{A_{h}^{-\frac{s}{2}} A_{0,h}^{\frac{s}{2}} P_h}{\cL(H)} \le C.
	\end{equation}
	
	We use the backward Euler method for the temporal discretization of $S$. For a time step ${\Delta t}\in(0,1]$ let $(t_j)_{j\in \N_0}\subset\R$ be given by $t_j={\Delta t} j$ and $N_{\Delta t} +1 = \inf\{j\in\N : t_j \notin [0,T]\}$. We write $S_{h,{\Delta t}}=(I +{\Delta t} A_h)^{-1}$. The discrete family $(S_{h,{\Delta t}}^j)_{j\in\{0,\dots,N_{\Delta t}\}}$ of powers of $S_{h,{\Delta t}}$ acts as a fully discrete approximation of~$S$. For brevity, we write $S_{h,{\Delta t}}^j$ for $S_{h,{\Delta t}}^j P_h$. For all $s \in [0,1)$, there is a constant $C<\infty$ such that for all $h,{\Delta t} \in (0,1]$ and $j = 1, \ldots, N_{\Delta t}$, $
	\norm{A_h^{s} S_{h,{\Delta t}}^j}{\cL(H)} \le C t_j^{-s}$,
	see~\cite[(8.7)]{FS91}. We define an interpolation $\tilde{S}_{h,\Delta t} \colon [0,T] \to \cL(\cH)$ of $(S_{h,{\Delta t}}^j)_{j\in\{0,\dots,N_{\Delta t}\}}$ by
	\begin{equation}
	\label{eq:semigroup_approximation_interpolation}
	\tilde{S}_{h,\Delta t}(t) = \chi_{\{0\}}(t) P_h + \sum_{j=1}^{N_{\Delta t}} \chi_{(t_{j-1},t_j]}(t) S^j_{h,\Delta t}, t \in [0,T],
	\end{equation}
	where $\chi$ denotes the indicator function. We immediately obtain that for all $s \in [0,1)$ there is a constant $C<\infty$ such that for all $h,{\Delta t} \in (0,1]$ and $t \in (0,T]$
	\begin{equation}
	\label{eq:advdiff:discsemistab}
	\norm{A_h^{s} \tilde{S}_{h,\Delta t}(t)}{\cL(H)} \le C t^{-s}.
	\end{equation} 
	
	The next lemma gives an error bound for the semigroup approximation. In the Dirichlet case, a proof is given in~\cite[Lemma~5.1]{AKL16}. In our general case, it is a consequence of the split 
	\begin{align*}
	&\norm{(S(t)-\tilde S_{h,\Delta t}(t))v}{H} \\
	&\quad\le \norm{(S(t)-S(t_j))v}{H} + \norm{(S(t_j)-S_{h}(t_j)P_h)v}{H} + \norm{(S_h(t_j)P_h-\tilde S_{h,\Delta t}(t_j))v}{H}
	\end{align*}
	for $t \in (t_{j-1},t_j]$, $j = 1, \ldots, N_{\Delta t}$, and $v \in H$. Here $S_h$ is the semigroup on $V_h$ generated by $A_h$ \cite[Section~7]{FS91}. The first term can be bounded by~\eqref{eq:advdiff:semistab}, \eqref{eq:advdiff:semiholder} and~\eqref{eq:advdiff:symnormequiv2} and the third by using an integral representation of the error as in the proof of \cite[Theorem~8.2]{FS91}. For the second term, one employs a similar interpolation argument as in~\cite[Lemma~5.1]{AKL16}, making use of~\cite[Theorem~7.1]{FS91} and~\cite[Lemma~3.8(iii)]{K14}. The latter result is, again, proven for Dirichlet boundary conditions, but the arguments are the same in our general case, cf.\ \cite[Remark~7.2]{FS91}.
	
	\begin{lemma}
		\label{thm:advdiff:semigrouperror}
		For all $\theta \in [0,2]$ and $s \in [0,2-\theta] \cap [0,1)$, there is a constant $C<\infty$ such that for all $h, \Delta t \in (0,1]$ and $t>0$
		\begin{equation*}
		\norm{S(t)-\tilde S_{h,\Delta t}(t)}{\cL(\dot{H}^{-s},H)}  = \norm{(S(t)-\tilde S_{h,\Delta t}(t))A_0^{\frac{s}{2}}}{\cL(H)} \le C (h^{\theta} + {\Delta t}^{\frac{\theta}{2}}) t^{-\frac{\theta + s}{2}}.
		\end{equation*}
	\end{lemma}
	
	We now define an approximation $(\tilde K_{h,\Delta t}(t_j))_{j=0}^{N_{\Delta t}}$ of~\eqref{eq:C} by $\tilde{K}_{h,\Delta t}(0) = P_h Q_\xi P_h$ and, for $j \ge 1$, 
	\begin{align*}
	&(I + \Delta t A_h) \tilde{K}_{h,\Delta t}(t_j) (I + \Delta t A_h)^* \\
	&\quad= \tilde{K}_{h,\Delta t}(t_{j-1}) + \Delta t F  \tilde{K}_{h,\Delta t}(t_{j-1}) + \Delta t \tilde{K}_{h,\Delta t}(t_{j-1}) F^* + \Delta t P_h B (P_h B)^*. %
	\end{align*}
	With our choice of $F$, this can equivalently be written as
	\begin{equation}
	\label{eq:advdiff:recursion}
	(I + \Delta t A_h) \tilde{K}_{h,\Delta t}(t_j) (I + \Delta t A_h)^* = (1 + 2 c_0 \Delta t) \tilde{K}_{h,\Delta t}(t_{j-1}) + \Delta t P_h Q P_h.
	\end{equation}
	To implement this scheme, we have to know the value of $c_0$ explicitly. Several choices are possible, see~\eqref{eq:advdiff:c0} for an example. A closed form of $\tilde K_{h,\Delta t}(t_j), j = 0, \ldots, N_{\Delta t},$ is given by
	\begin{equation}
	\label{eq:advdiff:discreteC}
	\begin{split}
	\tilde{K}_{h,\Delta t}(t_j) 
	&= \tilde{S}_{h,\Delta t}(t_j) Q_\xi \tilde{S}_{h,\Delta t}(t_j)^* + \int_{0}^{t_j}  \tilde{S}_{h,\Delta t}(t_j-s) F \tilde{K}_{h,\Delta t}(\dfloor{s})  \tilde{S}^*_{h,\Delta t}(t_j-s) \dd s \\
	&\quad+ \int_{0}^{t_j}  \tilde{S}_{h,\Delta t}(t_j-s) \tilde{K}_{h,\Delta t}(\dfloor{s}) ( \tilde{S}_{h,\Delta t}(t_j-s) F)^* \dd s \\ 
	&\quad + \int_{0}^{t_j} \tilde{S}_{h,\Delta t}(t_j-s) B ( \tilde{S}_{h,\Delta t}(t_j-s) B)^* \dd s.
	\end{split}
	\end{equation}
	Here $\dfloor{\cdot} = \floor{\cdot/\Delta t} \Delta t$, with $\floor \cdot$ denoting the floor function. The $\cL_1(H)$-norm of the last term can be bounded by~\eqref{eq:schatten_bound_2}, Assumption~\ref{ass:advdiff:ass1} and~\eqref{eq:advdiff:discsemistab} via 
	\begin{equation}
	\label{eq:advdiff:discreteCbound}
	\begin{split}
	&\Bignorm{\int_{0}^{t_j} \tilde{S}_{h,\Delta t}(t_j-s) B ( \tilde{S}_{h,\Delta t}(t_j-s) B)^* \dd s}{\cL_1(H)} \\
	&\quad\le\int_{0}^{t_j} \norm{\tilde{S}_{h,\Delta t}(t_j-s) B}{\cL_2(U_0,H)}^2 \dd s \\
	&\quad= \int_{0}^{t_j} \norm{\tilde{S}_{h,\Delta t}(t_j-s) A_h^{\frac{1-r}{2}}A_h^{\frac{r-1}{2}}A_{0,h}^{\frac{1-r}{2}}A_{0,h}^{\frac{r-1}{2}}P_h A_{0}^{\frac{1-r}{2}}A_{0}^{\frac{r-1}{2}} Q^{\frac{1}{2}} }{\cL_2(H)}^2 \dd s \\
	&\quad\lesssim \norm{A_h^{\frac{r-1}{2}}A_{0,h}^{\frac{1-r}{2}}}{\cL(H)}^2  \norm{A_{0,h}^{\frac{r-1}{2}} P_h A_{0}^{\frac{1-r}{2}}}{\cL(H)}^2 \norm{A_{0}^{\frac{r-1}{2}} Q^{\frac{1}{2}} }{\cL_2(H)}^2 \int^{t_j}_0 (t-s)^{r-1} \dd s \lesssim t^r \lesssim T^r.
	\end{split}
	\end{equation}
	Here we also made use of~\eqref{eq:advdiff:eqnorms2} and~\eqref{eq:advdiff:eqnorms3}. Using this result along with~\eqref{eq:advdiff:discsemistab} applied to the other terms of~\eqref{eq:advdiff:discreteC} we find that 
	\begin{equation*}
	\norm{\tilde{K}_{h,\Delta t}(t_j)}{\cL_1(H)} \lesssim 1 + \sum_{k = 0}^{j-1} \norm{\tilde{K}_{h,\Delta t}(t_k)}{\cL_1(H)}.
	\end{equation*}
	Along with the fact that $Q_\xi \in \cL_1(H)$, the discrete Gronwall lemma (see \cite[2.2~(9)]{G75}) implies that
	\begin{equation*}
	\sup_{h,\Delta t \in (0,1]} \sup_{j = 0, \ldots, N_{\Delta t}} \norm{\tilde{K}_{h,\Delta t}(t_j)}{\cL_1(H)} < \infty.
	\end{equation*}
	
	With this estimate in place, we move on to the main result of this section. 
	
	\begin{theorem}
		\label{thm:advdiff:error}
		Let Assumption~\ref{ass:advdiff:ass1} be satisfied. For all $\theta < 2r$, there is a constant $C<\infty$ such that for all $h, \Delta t \in (0,1]$
		\begin{equation*}
		\sup_{j = 0, \ldots, N_{\Delta t}} \norm{K(t_j) - \tilde{K}_{h,\Delta t}(t_j)}{\cL_1(H)} \le C (h^\theta + {\Delta t}^{\frac{\theta}{2}}).
		\end{equation*}
		Moreover, for all $\theta < 1+r$, there is a constant $C<\infty$ such that for all $h, \Delta t \in (0,1]$
		\begin{equation*}
		\sup_{j = 0, \ldots, N_{\Delta t}} \norm{K(t_j) - \tilde{K}_{h,\Delta t}(t_j)}{\cL_2(H)} \le C (h^\theta + {\Delta t}^{\frac{\theta}{2}}).
		\end{equation*}
	\end{theorem}
	
	\begin{proof}
		The result is a consequence of Proposition~\ref{prop:Kerr} applied to $\norm{K(t_j) - \tilde{K}_{h,\Delta t}(t_j)}{\cL_i(H)}$. We start with $i=1$.
		For the last term of the error decomposition in this proposition, the properties~\eqref{eq:schatten_bound_1}, \eqref{eq:schattenadjoint}, \eqref{eq:schatten_bound_2} and~\eqref{eq:advdiff:discsemistab} along with Assumption~\ref{ass:advdiff:ass1}, Lemma~\ref{thm:advdiff:semigrouperror} and an argument similar to that of~\eqref{eq:advdiff:discreteCbound} imply that
		\begin{equation}
		\label{eq:thm:advdiff:error:1}
		\begin{split}
		&\int^{t_j}_0 \norm{\cO^-(t_j-s) B (\cO^+(t_j-s)B)^*}{\cL_1(H)} \dd s \\
		&\quad \le \int^{t_j}_0 \norm{\cO^-(t_j-s) B}{\cL_2(U_0,H)} \norm{\cO^+(t_j-s)B}{\cL_2(U_0,H)} \dd s \\
		&\quad\lesssim \int^{t_j}_0 \norm{(S(t_j-s)-\tilde S_{h,\Delta t}(t_j-s))A_0^{\frac{1-r}{2}}}{\cL(H)} \norm{A_0^{\frac{r-1}{2}}Q^{\frac{1}{2}}}{\cL_2(H)}^2 \\
		&\hspace{5em}\times\left( \norm{\tilde{S}_{h,\Delta t}(t_j-s) A_h^{\frac{1-r}{2}} }{\cL(H)} + \norm{S(t_j-s) A^{\frac{1-r}{2}} }{\cL(H)} \right) \dd s \\
		&\quad\lesssim (h^\theta + {\Delta t}^\frac{\theta}{2}) \int_{0}^{t_j} (t_j-s)^{r-1-\frac{\theta}{2}} \dd s \lesssim T^{r-\frac{\theta}{2}} (h^\theta + {\Delta t}^\frac{\theta}{2}).
		\end{split}
		\end{equation}
		Similarly, the first, third and fourth term of the split in Proposition~\ref{prop:Kerr} can be bounded by a constant times $h^\theta + {\Delta t}^{\theta/2}$, noting that $F=c_0 I$. By~\eqref{eq:advdiff:semistab}, the second term is bounded by 
		\begin{align*}
		&\int^{t_j}_0 \norm{S(t_j-s) F (K(\dfloor{s})-\tilde K_{h,\Delta t}(\dfloor{s}))S(t_j-s)^*}{\cL_1(H)} \dd s \\ 
		&\quad\lesssim \Delta t \sum_{k = 0}^{j-1} \norm{ (K(t_k)-\tilde K_{h,\Delta t}(t_k))}{\cL_1(H)}.
		\end{align*}
		The proof for $i=1$ is now completed by an application of the discrete Gronwall inequality.
		
		In the case that $i=2$, the only major difference is the calculation in~\eqref{eq:thm:advdiff:error:1}. For this, let us note that since $\inpro[H]{v}{\cO^+(t_j-s) B u} = \inpro[U_0]{Q \cO^+(t_j-s)^* v}{u}$
		for $v \in H, u \in U_0$, we have $\cO^-(t_j-s) B (\cO^+(t_j-s)B)^* = \cO^-(t_j-s) Q \cO^+(t_j-s)^*$. The adjoint on the right hand side is taken with respect to $\cL(H)$. Using this observation, \eqref{eq:advdiff:discsemistab}, Lemma~\ref{thm:advdiff:semigrouperror} and the fact that $Q^{1/2} \in \cL(H)$ yields
		\begin{align*}
		&\int^{t_j}_0 \norm{\cO^-(t_j-s) B (\cO^+(t_j-s)B)^*}{\cL_2(H)} \dd s \\
		&\quad\le \int^{t_j}_0 \norm{\cO^-(t_j-s) Q^{\frac{1}{2}}}{\cL(H)} \norm{\cO^+(t_j-s)Q^{\frac{1}{2}}}{\cL_2(H)} \dd s \\
		&\quad\lesssim \int^{t_j}_0 \norm{(S(t_j-s)-\tilde S_{h,\Delta t}(t_j-s))}{\cL(H)} \norm{Q^{\frac{1}{2}}}{\cL(H)}   \\
		&\hspace{5em}\times\left( \norm{\tilde{S}_{h,\Delta t}(t_j-s) A_h^{\frac{1-r}{2}} }{\cL(H)} + \norm{S(t_j-s) A^{\frac{1-r}{2}} }{\cL(H)} \right) \norm{A_0^{\frac{r-1}{2}}Q^{\frac{1}{2}}}{\cL_2(H)} \dd s \\
		&\quad\lesssim (h^\theta + {\Delta t}^\frac{\theta}{2}) \int_{0}^{t_j} (t_j-s)^{\frac{r-1-\theta}{2}} \dd s \lesssim T^{\frac{r+1-\theta}{2}} (h^\theta + {\Delta t}^\frac{\theta}{2}). \qedhere
		\end{align*}
	\end{proof}
	
	\begin{example}
		\label{ex:heat}
		
		We conclude this section by demonstrating the results of Theorem~\ref{thm:advdiff:error} in the case that $\cD = (0,1)$ and $T = 1$. We choose deterministic initial conditions so that $Q_\xi=0$. Homogeneous Neumann boundary conditions are considered and we set $a_{1,1}(x)=4$, $a_1(x) = \sin(2 \pi x)$ and $a_0(x) = 0$ for all $x \in \cD$. With this, we may take $c_0 =1/8$. 
		From~\eqref{eq:advdiff:recursion}, we obtain a matrix recursion 
		\begin{equation*}
		(\mathbf{M}_{h} + \Delta t \mathbf{A}_{h}) \mathbf{K}_{j,h,\Delta t} (\mathbf{M}_{h} + \Delta t \mathbf{A}_h)^* = (1 + 2 c_0 \Delta t) \mathbf{M}_{h}\mathbf{K}_{j-1,h,\Delta t} \mathbf{M}_{h} + \Delta t \mathbf{Q}_{h}.
		\end{equation*}
		Here $\mathbf{K}_{j,h,\Delta t}$ is the matrix of coefficients $(k_{j,m,n})_{m,n=1}^{N_h}$ in the expansion
		\begin{equation*}
		\tilde{K}_{h,\Delta t}(t_j) = \sum_{m,n = 1}^{N_h} k_{j,m,n} \phi^h_m \otimes \phi^h_n,
		\end{equation*}
		where $(\phi^h_m)_{m=1}^{N_h}$ is the usual ''hat function'' basis of $V_h$ with $N_h = \dim(V_h)$. Furthermore, $(\mathbf{M}_{h})_{i,j} = \inpro[H]{\phi^h_i}{\phi^h_j}$, $(\mathbf{A}_{h})_{i,j} = a(\phi^h_i,\phi^h_j)$ and $(\mathbf{Q}_{h})_{i,j} = \inpro[H]{Q \phi^h_i}{\phi^h_j}$ for $i,j = 1, \ldots, N_h$. We solve this system of matrix equations for two choices of $Q$ and decreasing values of $h, \Delta t$.
		
		First, we consider the white noise case, i.e., $Q = I$. By our choice of $b$, the operator $A_0$ retains the Neumann boundary conditions of $A$. Lemma~2.3 in~\cite{KLP21b} then implies that Assumption~\ref{ass:advdiff:ass1} holds for all $r<1/2$. By Theorem~\ref{thm:advdiff:error}, we therefore expect to see a convergence rate essentially of order $1$ and $3/2$, respectively, if we plot the errors $\norm{K(T)-\tilde K_{h, \Delta t}(T)}{\cL_i(H)}$, $i \in \{1,2\}$, for $h = \sqrt{\Delta t} = 2^{-1},\ldots,2^{-7}$. This agrees with the results of~Figure~\ref{subfig:heat-errors-1}. 
		\begin{figure}[ht!]
			\centering		
			\subfigure[Errors with white noise.\label{subfig:heat-errors-1}]{\includegraphics[width = .49\textwidth]{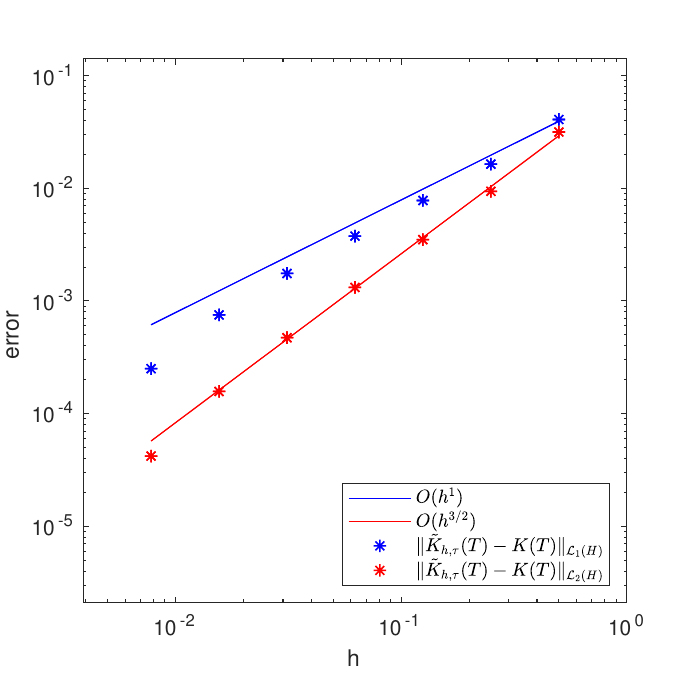}}
			\subfigure[Errors with exponential kernel noise. \label{subfig:heat-errors-2}]{\includegraphics[width = .49\textwidth]{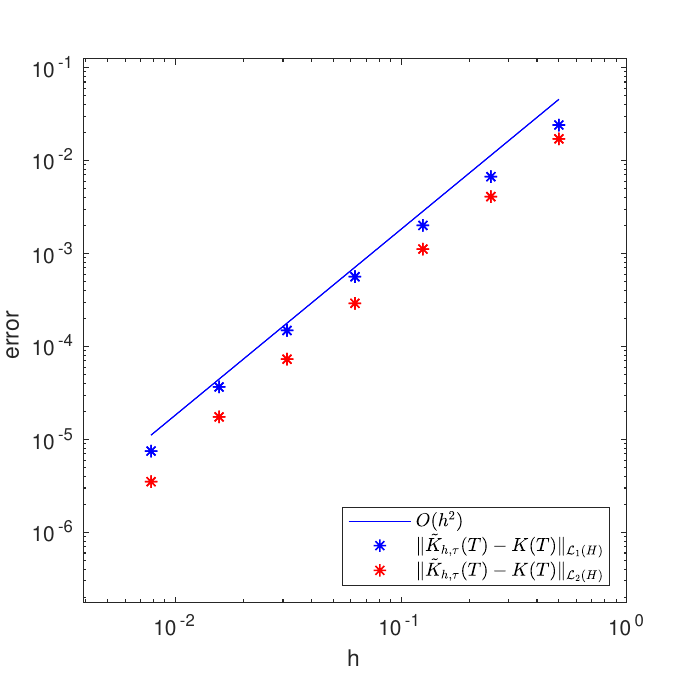}}
			\caption{Approximate errors $\norm{K(T)-\tilde K_{h, \Delta t}(T)}{\cL_i(H)}$, $i \in \{1,2\}$, for the equations of Example~\ref{ex:heat}.}
		\end{figure} 
		
		In place of $K(T)$ we used $\tilde K_{h', {\Delta t}'}(T)$ at $h' = \sqrt{{\Delta t}'} = 2^{-8}$. The errors were computed by
		\begin{equation}
		\label{eq:trace-error-formula}
		\norm{\tilde K_{h, {\Delta t}}(T)-\tilde K_{h', {\Delta t}'}(T)}{\cL_1(H)} = \trace\left(  \left|\mathbf{N}_{h,h'}^{\frac{1}{2}} \left[\begin{array}{cc}
		\mathbf{K}_{N_{\Delta t},h,\Delta t} & 0 \\
		0 & -\mathbf{K}_{N_{{\Delta t}'},h',{\Delta t}'}
		\end{array}\right] 	\mathbf{N}_{h,h'}^{\frac{1}{2}}\right| \right),
		\end{equation}
		where the absolute value of an operator $\Gamma$ is defined by $|\Gamma| = (\Gamma^*\Gamma)^{1/2}$, and
		\begin{equation}
		\label{eq:hs-error-formula}
		\begin{split}
		&\norm{\tilde K_{h, {\Delta t}}(T)-\tilde K_{h', {\Delta t}'}(T)}{\cL_2(H)}^2 \\ 
		&= \trace\left( \left(\mathbf{K}_{N_{\Delta t},h,\Delta t} \mathbf{M}_h\right)^2 \right) - 2 \trace\left(\mathbf{K}_{N_{\Delta t},h,\Delta t} \mathbf{M}_{h,h'} \mathbf{K}_{N_{{\Delta t}'},{h'},{\Delta t}'} \mathbf{M}_{h',h}\right)
		\\
		&\quad+ \trace\left(\left(\mathbf{K}_{N_{{\Delta t}'},{h'},{\Delta t}'} \mathbf{M}_{h'}\right)^2 \right).
		\end{split}
		\end{equation}
		Here $(\mathbf{M}_{h,h'})_{i,j} = \inpro[H]{\phi^h_i}{\phi^{h'}_j}$ for $i = 1, \ldots, N_h$, $j = 1, \ldots, N_{h'}$ and
		\begin{equation*}
		\mathbf{N}_{h,h'} = \left[\begin{array}{cc}
		\mathbf{M}_{h} & \mathbf{M}_{h,h'} \\
		\mathbf{M}_{h',h} & \mathbf{M}_{h'}
		\end{array}\right].
		\end{equation*}
		The formula~\eqref{eq:hs-error-formula} is a straightforward consequence of the definition of the Hilbert--Schmidt norm while~\eqref{eq:trace-error-formula} requires a more detailed argument. We postpone this to an appendix.

		Now, we consider the case that $Q$ is an integral operator defined by
		\begin{equation*}
		\inpro[H]{Qu}{v} = \int_{\cD \times \cD} q(x,y) u(x) v(y) \dd x \dd y, u,v \in H.
		\end{equation*} 
		We set $q(x,y) = \exp(-2|x-y|)$ for $x,y \in \cD$. Figure~\ref{subfig:advdiff-path} shows an approximate realization of the solution to~\eqref{eq:advdiff} for this $q$. Figure~\ref{subfig:advdiff-cov} shows its covariance function, corresponding to~$\tilde K_{h, {\Delta t}}(t)$, at $t=0.1$. Since $\trace(Q) < \infty$, Assumption~\ref{ass:advdiff:ass1} holds for $r\le1$, so we expect a rate of order $2$ for both norms, in the same setup as before. Again, this agrees with Figure~\ref{subfig:heat-errors-2}. 
	\end{example}
	
	\begin{remark}
		\label{rem:heatstrongvcov}
		If one would instead directly compute an approximation $\tilde X_{h,\Delta t}$ based on the same semigroup approximation $\tilde S_{h,\Delta t}$ and compute its covariance by the method of \cite{P20}, we would directly get a bound on the covariance error $\norm{K(T)-\Cov(\tilde X_{h,\Delta t}(T))}{\cL_i(H)}$ by the strong error $\norm{X(t)-\tilde X_{h,\Delta t}(T)}{L^2(\Omega,H)}$, see \eqref{eq:introcovboundedbystrong}. Note that when $c_0 = 0$, $\Cov(\tilde X_{h,\Delta t}(T)) = \tilde K_{h, \Delta t}(T)$. In the first case above, when $Q=I$, this means that the error would be bounded by $h^r + k^{r/2}$ for $r < 1/2$, which is a lower rate than the estimates obtained above for both $i = 1$ and $i = 2$. See~\cite{Y05} for the calculations that yield this rate. Let us also note that our covariance error rate in the case that $i=1$ coincides with that of the weak error $\left|\E\left[\phi\big(X(T)\big) - \phi\big(\tilde X_{h,\Delta t}(T)\big)\right]\right|$,
		where $\phi: H \to \R$ is a smooth functional, see \cite{KLL13}. For the case $i = 2$, the covariance error rate exceeds the weak error rate.
	\end{remark}
	
	\subsection{The stochastic wave equation}
	\label{sec:wave}
	
	Next, we apply our results to the stochastic wave equation perturbed by a linear inhomogeneity $G$, formally given by
	\begin{equation}
	\label{eq:original_wave_equation}
	\dd \dot{U} (t,x) - \Delta U(t,x) \dd t =  G U(t,x) \dd t + \dd W(t,x), t \in [0,T], x \in \cD.
	\end{equation}
	Here $\dot{U}$ denotes the first time derivative of the unknown $H= L^2(\cD)$-valued stochastic process~$U$ with initial conditions $\dot{U}(0) =\psi$ and $U(0)=\chi$. With $G=-Q$ we obtain the DNA model~\eqref{eq:wave-intro} of the introduction. We again let $W$ be a generalized Wiener process in $H$ with covariance operator $Q \in \Sigma^+(H)$. We introduce precise assumptions on $G,Q, u_0$ and $v_0$ below and write $\Lambda = - \Delta$ for the negative Laplacian with zero Dirichlet boundary conditions on $H$. The spaces $\dot{H}^\alpha$, along with fractional powers $\Lambda^{\alpha/2}$, can be defined as for the operator $A_0$ in Section~\ref{sec:heat}.
	
	In order to treat~\eqref{eq:original_wave_equation} in a semigroup framework, we define for $\alpha \in \R$ the Hilbert space $\cH^\alpha = \dot{H}^\alpha \oplus \dot{H}^{\alpha-1}$ with inner product $\inpro[\cH^\alpha]{v}{w} = \inpro[\dot{H}^\alpha]{v_1}{w_1} + \inpro[\dot{H}^{\alpha-1}]{v_2}{w_2}$ for $v = [v_1,v_2]^\top, w = [w_1,w_2]^\top \in \cH^\alpha$. Writing $\cH = \cH^0$, let $A\colon \dom(A) = \cH^1 \to \cH$, $B\colon \dot{H}^{-1} \to \cH$ 
	and $\Theta^{\alpha/2}\colon \cH^\alpha \to \cH, \alpha \in \R,$ be given by 
	\begin{equation*}
	A = \left[\begin{array}{cc}
	0 &- I \\
	\Lambda &0
	\end{array}\right],
	B = \left[
	\begin{array}{c}
	0 \\
	I
	\end{array}\right]
	\text{ and }
	\Theta^{\frac{\alpha}{2}} = \left[\begin{array}{cc}
	\Lambda^{\frac{\alpha}{2}} &0 \\
	0 & \Lambda^{\frac{\alpha}{2}}
	\end{array}\right].
	\end{equation*}
	The third operator is used to relate the norms of $\cH^\alpha$ and $\cH$ via $\norm{\cdot}{\cH^\alpha} = \norm{\Theta^{\alpha/2} \cdot}{\cH}$. Note that, since $\Lambda^{\alpha/2}$ extends to an operator in $\cL(\dot{H}^s,\dot{H}^{s-\alpha})$ for all $\alpha, s \in \R$, so $\Theta^{\alpha/2}$ 
	can be extended to an operator in $\cL(\cH^s,\cH^{s-\alpha})$. We do so without changing notation. We write $P_1$ for the projection given by $P_1 v = v_1$ for $v = [v_1, v_2]^\top \in \cH$. Note that $\Theta^{\alpha/2} B = B \Lambda^{\alpha/2}$. From this, we obtain the identities $
	\norm{\Theta^{\frac{\alpha}{2}} B v}{\cH} = \norm{\Lambda^{\frac{\alpha-1}{2}}v}{H} = \norm{v}{\dot{H}^{\alpha-1}}$,
	$v \in \dot{H}^{\alpha-1}$, and
	\begin{equation}
	\label{eq:B_bound_2}
	\norm{B}{\cL(\dot{H}^{-1},\cH)} = 
	\norm{B\Lambda^\frac{1}{2}}{\cL(H,\cH)} = \norm{\Theta^{\frac{1}{2}}B}{\cL(H,\cH)} = 1.
	\end{equation}
	
	The operator $-A$ is the generator of a $C_0$-semigroup (actually a group, see \cite{L12}) on $\cH$ given by
	\begin{equation*}
	S(t) =
	\left[\begin{array}{cc}
	\cos(t \Lambda^{\frac{1}{2}}) & \Lambda^{-\frac{1}{2}} \sin(t \Lambda^{\frac{1}{2}}) \\
	-\Lambda^{\frac{1}{2}} \sin(t \Lambda^{\frac{1}{2}}) & \cos(t \Lambda^{\frac{1}{2}})
	\end{array}\right] \text{ for }t \in \R.
	\end{equation*}
	It satisfies $\Theta^{\frac{\alpha}{2}}S(\cdot) = S(\cdot) \Theta^{\frac{\alpha}{2}}$ and $\Lambda^{\frac{\alpha}{2}}P_1 S(\cdot)B = P_1 S(\cdot)B \Lambda^{\frac{\alpha}{2}}$ for all $\alpha \in \R$. Moreover,
	\begin{equation}
	\label{eq:semigroup_bound} 
	\norm{S(t)}{\cL(\cH)} \le 1  \text{ for }t \in \R.
	\end{equation}
	
	If we set $X(t) = [X_1(t), X_2(t)]^\top = [U(t,\cdot), \dot{U}(t,\cdot)]^\top$, \eqref{eq:original_wave_equation} can be put in an abstract It\^o form
	\begin{equation*}
	\dd X(t) + A X(t) \dd t = F X(t) \dd t + B \dd W(t), t \in (0,T],
	\end{equation*}
	with $F = B G P_1$. We make the following assumptions on $G \colon H \to \dot{H}^{-1}$, $Q \in \Sigma^+(H)$ and the initial condition $X(0) = \xi = [\chi, \psi]^\top$.
	\begin{assumption} 
		\label{assumptions:waveexistence} There is a constant $C < \infty$ such that
		\begin{enumerate}[label=(\roman*)]
			\item %
			$\norm{G}{\cL(H,\dot{H}^{-1})} = \norm{\Lambda^{-1/2} G}{\cL(H)} \le C$ and
			\item \label{assumptions:waveexistence:Q}
			$\norm{Q^{1/2}}{\cL_2(H,\dot{H}^{-1})} = \norm{\Lambda^{-1/2} Q^{1/2}}{\cL_2(H)} \le C$.
		\end{enumerate}
		Moreover, $\xi$ is an $\cH$-valued $\cF_0$-measurable Gaussian random variable. 
	\end{assumption}
	
	With this assumption in place, we can put the equation in the framework of Section~\ref{sec:covariance}. Combining~\eqref{eq:B_bound_2} with~\eqref{eq:semigroup_bound} and Assumption~\ref{assumptions:waveexistence}, we find that $\norm{S(t)F}{\cL(\cH)}$ and $\norm{S(t)B}{\cL_2(U_0,\cH)}$ are bounded on $[0,T]$. By the strong continuity of  $S$, Assumption~\ref{assumptions:1} is fulfilled. 
	Therefore, we can apply Propositions~\ref{prop:Kexists} and~\ref{prop:Xexists} along with Theorem~\ref{thm:covarianceisK} to obtain a unique solution $K$ to~\eqref{eq:C} such that $K(t) = \Cov(X(t))$, $t \in [0,T]$.
	
	We now move on to approximations of $K$. We first consider temporally semidiscrete approximations based on a rational approximation of $S$. This is a stepping stone towards analyzing a fully discrete approximation. For a time step ${\Delta t}\in(0,1]$ we again let $(t_j)_{j\in \N_0}\subset\R$ be given by $t_j=j \Delta t$ and $N_{\Delta t} +1 = \inf\{j\in\N : t_j \notin [0,T]\}$. Let $R\colon \C \to \C$ be a rational function such that $|R(iy)|\le1$ for all $y \in \R$ and, for some approximation order $\rho \in \N$ and $C,b > 0$, $|R(iy)-e^{-iy}| \le C |y|^{\rho+1}$ for all $y \in \R$ with $|y|\le b$, where $i = \sqrt{-1}$. We write $S^n_{\Delta t} = R(\Delta t A)^n$ for the rational approximation of the operator $S(t_n)$. An interpolation $\tilde{S}_{\Delta t} \colon [0,T] \to \cL(\cH)$ of the approximation is defined as in~\eqref{eq:semigroup_approximation_interpolation}, and we say that $\tilde{S}_{\Delta t}$ approximates the semigroup $S$ with order $\rho \in \N$.
	This framework covers many well-known temporal discretizations, cf.\ \cite[Section~4]{BB79}. Examples include the backward Euler scheme of the previous section (with order $\rho = 1$) as well as the Crank--Nicolson approximation (with order $\rho=2$) given by $R(\Delta t A) = (1 - \Delta t A/2)(1 + \Delta t A/2)^{-1}$.
	The stability result
	\begin{equation}
	\label{eq:wave_semidiscrete_stability}
	\norm{\tilde{S}_{\Delta t}(t)}{\cL(\cH)} \le 1,
	\end{equation} 
	for all $t \in [0,T]$, holds uniformly in $\Delta t \in (0,1]$ \cite[Section~4.2]{KLL13}. Moreover, $\Theta^{\alpha/2}$ commutes with $\tilde{S}_{\Delta t}$ for all $\alpha \in \R$ and we have the following error estimate:
	\begin{lemma}[{\cite[Lemma~4.4]{KLL13}}]
		\label{lem:wave_semidisc_semigroup_error}
		Assume that $\tilde{S}_{\Delta t}$ approximates~$S$ with order $\rho \in \N$. Then, for each $\alpha \ge 0$, there is a constant $C < \infty$ such that for all $\Delta t \in (0,1]$,
		\begin{align*}
		\sup_{t \in [0,T]} \norm{(\tilde{S}_{\Delta t}(t)-S(t))\Theta^{-\frac{\alpha}{2}}}{\cL(\cH)} \le C {\Delta t}^{\min(\alpha \frac{\rho}{\rho+1},1)}.
		\end{align*}
	\end{lemma}
	
	We define a semidiscrete approximation $\tilde K_{\Delta t}$ of $K$ by $\tilde K_{\Delta t}(0) = K(0) = Q_\xi$ and, for $j \ge 1$,
	\begin{equation}	
	\label{eq:wave_semi_rec}
	\begin{split}
	\tilde K_{\Delta t}(t_j) &= \tilde S_{\Delta t}\tilde K_{\Delta t}(t_{j-1})\tilde S_{\Delta t}^* \\
	&+ \Delta t \tilde S_{\Delta t} F  \tilde K_{\Delta t}(t_{j-1}) \tilde S_{\Delta t}^* + \Delta t \tilde S_{\Delta t} \tilde K_{\Delta t}(t_{j-1}) (\tilde S_{\Delta t} F)^* + \Delta t \tilde S_{\Delta t}B (\tilde S_{\Delta t}B)^*.
	\end{split}
	\end{equation}
	This definition is then extended to arbitrary $t \in [0,T]$ by
	\begin{align}
	\label{eq:wave_semi_K}
	\tilde K_{\Delta t}(t) &= \tilde S_{\Delta t}(t) Q_\xi \tilde S_{\Delta t}^*(t) \notag \\
	&\quad+ \int_{0}^{t}  \tilde S_{\Delta t}(t-s) F \tilde K_{\Delta t}(\dfloor{s})  \tilde S_{\Delta t}^*(t-s)+  \tilde S_{\Delta t}(t-s) \tilde K_{\Delta t}(\dfloor{s}) ( \tilde S_{\Delta t}(t-s) F)^* \dd s \\ 
	&\quad + \int_{0}^{t} \tilde S_{\Delta t}(t-s) B ( \tilde S_{\Delta t}(t-s) B)^* \dd s. \notag
	\end{align}
	An additional assumption provides this approximation with some spatial regularity.
	
	\begin{assumption} 
		\label{assumptions:waveregularity} There are constants $C<\infty$ and $r\ge0$ such that, for some $i \in \{1,2\}$,
		\begin{enumerate}[label=(\roman*)]
			\item \label{assumptions:waveregularity:G}
			$\norm{G}{\cL(\dot{H}^{r},\dot{H}^{r-1})} = \norm{\Lambda^{(r-1)/2} G \Lambda^{-r/2}}{\cL(H)} \le C$,
			\item \label{assumptions:waveregularity:Q}
			$\norm{Q}{\cL_i(\dot{H}^{1},\dot{H}^{r-1})} = \norm{\Lambda^{(r-1)/2} Q \Lambda^{-1/2}}{\cL_i(H)} \le C$, and
			\item \label{assumptions:waveregularity:initialcov} $\norm{Q_\xi}{\cL_i(\cH,\cH^r)} = \norm{\Theta^{r/2} Q_\xi}{\cL_i(\cH,\cH)} \le C$.
		\end{enumerate}
	\end{assumption}
	
	Before deriving the spatial regularity result in Lemma~\ref{lem:wave-regularity} below, we make some comments on these requirements. If Assumption~\ref{assumptions:waveregularity}\ref{assumptions:waveregularity:Q}-\ref{assumptions:waveregularity:initialcov} are satisfied for $i=1$ with a particular parameter $r > 0$, then they are also satisfied for $i=2$ with the same parameter $r$. Moreover, for $i=1$ and $r>0$, Assumption~\ref{assumptions:waveregularity}\ref{assumptions:waveregularity:Q} is strictly stronger than Assumption~\ref{assumptions:waveexistence}\ref{assumptions:waveexistence:Q} \cite[Lemma~4.1]{KLL13}.
	Assumption~\ref{assumptions:waveregularity}\ref{assumptions:waveregularity:initialcov} is fulfilled if the random variable $\xi$ takes values in $\cH^r$, in particular if $\chi$ and $\psi$ with $\xi = [\chi,\psi]^\top$ are jointly Gaussian (taking values in $\dot{H}^r$ and $\dot{H}^{r-1}$, respectively) or if they are deterministic. To see this, note first that if $\xi$ is a Gaussian $\cH^r$-valued random variable, then it is also Gaussian in $\cH$. Write $Q_{\xi,r}$ for the covariance of $\xi$ in~$\cH^r$. This operator is related to the covariance of $\xi$ in~$\cH$ by $Q_\xi = Q_{\xi,r} \Theta^{-r}$.
	By~\eqref{eq:schatten_bound_1}, 
	\begin{align*}
	\norm{Q_\xi}{\cL_1(\cH^{-r},\cH^{r})} &= \norm{Q_{\xi,r} \Theta^{-r}}{\cL_1(\cH^{-r},\cH^{r})} \le \norm{Q_{\xi,r}}{\cL_1(\cH^{r},\cH^{r})} \norm{\Theta^{-r}}{\cL(\cH^{-r},\cH^{r})} \\
	&= \norm{Q_{\xi,r}}{\cL_1(\cH^{r},\cH^{r})} < \infty.
	\end{align*}
	\begin{lemma}
		\label{lem:wave-regularity}
		Let Assumptions~\ref{assumptions:waveexistence} and~\ref{assumptions:waveregularity} be satisfied for some $r \ge 0$ and $i \in \{1,2\}$ and let $\tilde{S}_{\Delta t}$ approximate~$S$ with order $\rho \in \N$. Then, there is a constant $C < \infty$ such that for all $\Delta t \in (0,1]$
		\begin{equation*}
		\sup_{t \in [0,T]} \norm{\tilde K_{\Delta t}(t)}{\cL_i(\cH,\cH^r)} = \sup_{t \in [0,T]} \norm{\Theta^{\frac{r}{2}} \tilde K_{\Delta t}(t)}{\cL_i(\cH)} \le C.
		\end{equation*}
	\end{lemma}
	\begin{proof}
		By applying the triangle inequality to~\eqref{eq:wave_semi_K}, we obtain, for $n = 0, \ldots, N_{\Delta t}$,
		\begin{equation*}
		\norm{\Theta^{\frac{r}{2}} \tilde K_{\Delta t}(t_n)}{\cL_i(\cH)} \le \mathrm{I} + \Delta t \sum^{n-1}_{j=0} \left(\mathrm{II}_j + \mathrm{III}_j + \mathrm{IV}_j\right),
		\end{equation*}
		where we have written $\mathrm{I} = \norm{\Theta^{r/2} \tilde S_{\Delta t}^n Q_\xi (\tilde S_{\Delta t}^n)^*}{\cL_i{(\cH)}}$, $\mathrm{II}_j = \norm{\Theta^{r/2} \tilde S_{\Delta t}^{n-j} F K(t_j) (\tilde S_{\Delta t}^{n-j})^*}{\cL_i{(\cH)}}$, $\mathrm{III}_j = \norm{\Theta^{r/2} \tilde S_{\Delta t}^{n-j} \tilde K_{\Delta t}(t_j) (\tilde S_{\Delta t}^{n-j} F)^*}{\cL_i{(\cH)}}$ and $\mathrm{IV}_j = \norm{\Theta^{r/2} \tilde S_{\Delta t}^{n-j} B (\tilde S_{\Delta t}^{n-j} B)^*}{\cL_i{(\cH)}}$. For the first term, the commutativity of $\Theta^{r/2}$ and $\tilde S_{\Delta t}^n$ along with~\eqref{eq:wave_semidiscrete_stability} and~Assumption~\ref{assumptions:waveregularity}\ref{assumptions:waveregularity:initialcov} yields the existence of a constant $C$, that does not depend on $\Delta t $ or $n$, such that $\mathrm{I} \le \norm{\tilde S_{\Delta t}^n}{\cL(\cH)}^2 \norm{\Theta^{\frac{r}{2}} Q_\xi}{\cL_i(\cH)} \le C$.
		Similarly, using also~\eqref{eq:B_bound_2}, the fact that $\Theta^{r/2} B = B \Lambda^{r/2}$ and Assumption~\ref{assumptions:waveregularity}\ref{assumptions:waveregularity:G}, we see that 
		\begin{align*}
		\mathrm{II}_j &= \norm{\tilde S_{\Delta t}^{n-j} \Theta^{\frac{1}{2}}B \Lambda^{\frac{r-1}{2}} G P_1 K(t_j) (\tilde S_{\Delta t}^{n-j})^*}{\cL_i{(\cH)}}\\
		&\le \norm{\tilde S_{\Delta t}^{n-j}}{\cL{(\cH)}}^2 \norm{\Theta^{\frac{1}{2}}B}{\cL(H,\cH)} \norm{\Lambda^{\frac{r-1}{2}} G P_1 \tilde K_{\Delta t}(t_j)}{\cL_i{(\cH,H)}} \\
		&\lesssim \norm{\Lambda^{\frac{r-1}{2}} G \Lambda^{-\frac{r}{2}}}{\cL(H)} \norm{ \Lambda^{\frac{r}{2}} P_1 \tilde K_{\Delta t}(t_j)}{\cL_i{(\cH,H)}} \lesssim \norm{\Theta^{\frac{r}{2}} \tilde K_{\Delta t}(t_j)}{\cL_i{(\cH,H)}}.
		\end{align*}
		The term $\mathrm{III}_j$ is treated in the same way. Finally, we note that 
		\begin{equation*}
		\mathrm{IV}_j = \norm{\Theta^{\frac{r}{2}} \tilde S_{\Delta t}^{n-j} B (\tilde S_{\Delta t}^{n-j} B)^*}{\cL_i{(\cH)}} = \norm{\Theta^{\frac{r}{2}} \tilde S_{\Delta t}^{n-j} B Q B^* (\tilde S_{\Delta t}^{n-j})^*}{\cL_i{(\cH)}},
		\end{equation*} 
		where the adjoint of the expression involving $B$ is taken with respect to $\cL(U_0,\cH)$ on the left hand side of the last equality and with respect to $\cL(H,\cH)$ on the right hand side. Using~\eqref{eq:schatten_bound_1}, \eqref{eq:B_bound_2}, \eqref{eq:wave_semidiscrete_stability} and Assumption~\ref{assumptions:waveregularity}\ref{assumptions:waveregularity:Q} yields the existence of a constant $C < \infty$, independent of $\Delta t$, $n$ and $j$, such that
		\begin{equation*}
		\mathrm{IV}_j\lesssim \norm{\tilde S_{\Delta t}^{n-j}}{\cL{(\cH)}}^2 \norm{B \Lambda^\frac{1}{2}}{\cL{(H,\cH)}}^2 \norm{\Lambda^{\frac{r-1}{2}} Q \Lambda^{-\frac{1}{2}}}{\cL_i(H)} \le C.
		\end{equation*}
		By an appeal to the discrete Gronwall lemma, we find that there is a constant $C < \infty$, independent of $\Delta t$, such that $\norm{\Theta^{\frac{r}{2}} \tilde K_{\Delta t}(t_n)}{\cL_i(\cH)} \le C$ for $n = 0, \ldots, N_{\Delta t}$. The supremum bound over $[0,T]$ follows by using this result in the representation~\eqref{eq:wave_semi_K}.
	\end{proof}
	
	Proposition~\ref{prop:Kerr}, Lemma~\ref{lem:wave_semidisc_semigroup_error} and Lemma~\ref{lem:wave-regularity} yield a semidiscrete error bound.
	
	\begin{theorem}
		\label{thm:wavesemidiscreteerror}
		Let Assumptions~\ref{assumptions:waveexistence} and~\ref{assumptions:waveregularity} be satisfied for some $r \ge 0$ and $i \in \{1,2\}$. Let $\tilde{S}_{\Delta t}$ approximate $S$ with order $\rho \in \N$. Then, there is a constant $C < \infty$ such that for all $\Delta t \in (0,1]$
		\begin{equation*}
		\sup_{t \in [0,T]} \norm{K(t)-\tilde K_{\Delta t}(t)}{\cL_i(\cH)} \le C {\Delta t}^{\min(r\frac{\rho}{\rho+1},1)}.
		\end{equation*}
	\end{theorem}
	\begin{proof}
		We write $\mathrm{I}, \mathrm{II}, \ldots, \mathrm{V}$ for the five terms appearing in the error decomposition of $\norm{K(t_n)-\tilde K_{\Delta t}(t_n)}{\cL_i(\cH)}$ in~Proposition~\ref{prop:Kerr}, with $n = 0, \ldots, N_{\Delta t}$. From Lemma~\ref{lem:wave_semidisc_semigroup_error}, \eqref{eq:wave_semidiscrete_stability}, \eqref{eq:semigroup_bound} and the commutativity of $S(t_n), \tilde{S}_{\Delta t}(t_n)$ and $\Theta^{r/2}$ we find that $\mathrm{I}$ is bounded by
		\begin{equation*}
		\norm{(\tilde{S}_{\Delta t}(t_n)-S(t_n))\Theta^{-\frac{r}{2}}}{\cL(\cH)} \norm{\Theta^{\frac{r}{2}} Q_\xi}{\cL_i(\cH)} \left(\norm{\tilde{S}_{\Delta t}(t_n)}{\cL(\cH)} + \norm{S(t_n)}{\cL(\cH)}\right) \lesssim {\Delta t}^{\min(r\frac{\rho}{\rho+1},1)}.
		\end{equation*}
		For the next three terms, we also use Assumption~\ref{assumptions:waveregularity}\ref{assumptions:waveregularity:G} (with $r = 0$) along with~\eqref{eq:B_bound_2} and~\eqref{eq:semigroup_bound} to see that
		\begin{equation*}
		\mathrm{II} \lesssim  \norm{B\Lambda^{\frac{1}{2}}}{\cL(H,\cH)} \norm{\Lambda^{-\frac{1}{2}}G}{\cL(H)} \Delta t \sum_{j = 0}^{n-1} \norm{K(t_j)-\tilde K_{\Delta t}(t_j)}{\cL_i(\cH)}
		\end{equation*}
		and that the terms $\mathrm{III}$ and $\mathrm{IV}$ may be bounded by a constant multiplied by
		\begin{equation*}
		\sum_{j = 0}^{n-1} \int_{t_j}^{t_{j+1}} \norm{(\tilde{S}_{\Delta t}(t_n-s)-S(t_n-s))\Theta^{-\frac{r}{2}}}{\cL(\cH)} \norm{\Theta^{\frac{r}{2}}\tilde K_{\Delta t}(t_j)}{\cL_i(\cH)} \dd s \lesssim {\Delta t}^{\min(r\frac{\rho}{\rho+1},1)}, 
		\end{equation*}
		where Lemmas~\ref{lem:wave_semidisc_semigroup_error} and~\ref{lem:wave-regularity} were used. The last term is treated like $\mathrm{IV}_j$ in Lemma~\ref{lem:wave-regularity}, so
		\begin{align*}
		\mathrm{V} &\le  \int_{0}^{t_{n}} \norm{(\tilde{S}_{\Delta t}(t_n-s)-S(t_n-s))\Theta^{-\frac{r}{2}}}{\cL(\cH)} 
		\norm{B \Lambda^\frac{1}{2}}{\cL{(H,\cH)}}^2
		\norm{\Lambda^{\frac{r-1}{2}} Q \Lambda^{-\frac{1}{2}}}{\cL_i(H)} \dd s \\
		&\lesssim {\Delta t}^{\min(r\frac{\rho}{\rho+1},1)}.
		\end{align*}
		The proof is completed by first using the discrete Gronwall lemma, and then extending the resulting bound to $[0,T]$, as in the proof of Lemma~\ref{lem:wave-regularity}. 
	\end{proof}
	
	Next, this result is used in a convergence analysis of a fully discrete approximation to~\eqref{eq:C}. Let $(V^\kappa_h)_{h \in (0,1]} \subset \dot{H}^1$, $\kappa \in \{2,3\}$ be a standard family of finite element function spaces consisting of continuous piecewise polynomials of degree $\kappa-1$, with respect to a regular family of triangulations of $\cD$ with maximal mesh size $h$, that are zero on the boundary of~$\cD$. They are equipped with the inner product $\inpro[V_h^\kappa]{\cdot}{\cdot} = \inpro[H]{\cdot}{\cdot}$. On this space, let a discrete counterpart $\Lambda_h\colon V^\kappa_h \to V^\kappa_h$ to $\Lambda$ be defined by $\inpro[H]{\Lambda_h v_h}{u_h} = \inpro[H]{\Lambda^{\frac{1}{2}} v_h}{\Lambda^{\frac{1}{2}}u_h} = \inpro[\dot{H}^1]{v_h}{u_h}$
	for all $v_h, u_h \in V^\kappa_h$. By $P_h \colon \dot{H}^{-1} \to V^\kappa_h$ we denote the generalized orthogonal projector. We define $\cV_h^\kappa = V^\kappa_h \oplus V^\kappa_h$, equipped with the same inner product as $\cH$. With some abuse of notation, by the expression $P_h v$, $v = [v_1, v_2]^\top \in \cH$, we denote the element $[P_h v_1, P_h v_2]^\top \in \cV_h^\kappa$. Let
	\begin{equation*}
	A_h = \left[\begin{array}{cc}
	0 &- I \\
	\Lambda_h &0
	\end{array}\right]
	\end{equation*}
	be a discrete counterpart to $A$ on $\cV_h^\kappa$ and set $S^k_{h,\Delta t} = R(\Delta t A_h)^k P_h$ for $k = 0, \ldots, n$ with a step function extension $\tilde{S}_{h,\Delta t} \colon [0,T] \to \cL(\cH)$ defined in the same way as for $\tilde{S}_{\Delta t}$. The stability result $\norm{\tilde{S}_{h,\Delta t}(t)}{\cL(\cH)} \le 1$, 
	$t \in [0,T]$, holds uniformly in $h, \Delta t \in (0,1]$, see \cite{KLL13}.
	
	One could consider defining a fully discrete version of $K$ in~\eqref{eq:C} by the semigroup approximation $\tilde{S}_{h,\Delta t}$ directly, obtaining from~\eqref{eq:HSlinops} an approximation of $\Cov(U)$ via
	$\Cov(U(t)) = P_1 \Cov(X(t))P_1^*$. The problem is that we only have access to an error bound in the first component of $\tilde{S}_{h,\Delta t}$ (see the next lemma) which means that we cannot use a Gronwall argument as in Theorem~\ref{thm:wavesemidiscreteerror}. Instead, we employ the approximation $\tilde{S}_{h,\Delta t}$ in an indirect way. 
	\begin{lemma}[{\cite[Corollary~4.2]{KLL13}}]
		\label{lem:wave_fullydisc_semigroup_error}
		Let $\tilde{S}_{h, \Delta t}$ approximate $S$ with order $\rho \in \N$ in time and $\kappa \in \{2,3\}$ in space. Then, for each $\alpha \ge 0$, there is a constant $C<\infty$ such that, for all $h, \Delta t \in (0,1]$,
		\begin{align*}
		\sup_{t \in [0,T]} \norm{P_1 (\tilde{S}_{h, \Delta t}(t)-S(t))\Theta^{-\frac{\alpha}{2}}}{\cL(\cH,H)} \le C \left({\Delta t}^{\min(\alpha \frac{\rho}{\rho+1},1)} + h^{\min(\alpha \frac{\kappa}{\kappa+1},\kappa)} \right).
		\end{align*}
	\end{lemma}
	
	We now define a perturbed temporally semidiscrete semigroup approximation $\hat{S}_{\Delta t} \colon [0,T] \to \cL(\cH)$ by a step function extension (as in~\eqref{eq:semigroup_approximation_interpolation}) of $(R(\Delta t A) (I+{\Delta t}F))^n$, $n = 0, \ldots, N_{\Delta t}$. Similarly we define a perturbed fully discrete semigroup approximation $\hat{S}_{h, \Delta t} \colon [0,T] \to \cL(\cH)$ using $(R(\Delta t A_h) P_h (I+{\Delta t}F))^n $. Note that, with $t_j \in [0,T]$, $u \in \cH$, $\tilde S_{\Delta t}(t_j) u$ and $\tilde{S}_{h, \Delta t}(t_j) u$ are nothing but semidiscrete and fully discrete approximations of $X(t_j)$ with initial value $\xi = u$ and noise covariance $Q=0$. Using this, the following lemma is proven by a  standard Gronwall argument as in, e.g., \cite[Theorem~3.5]{KLP20}, making use of Lemmas~\ref{lem:wave_semidisc_semigroup_error} and Lemma~\ref{lem:wave_fullydisc_semigroup_error}. 	
	
	\begin{lemma}
		\label{lem:wave_fullydisc_semidisc_error}
		Let Assumption~\ref{assumptions:waveregularity} be satisfied. Let $\hat{S}_{\Delta t} \colon [0,T] \to \cL(\cH)$ and $\hat{S}_{h, \Delta t} \colon [0,T] \to \cL(\cH)$ be step-function extensions of $(R(\Delta t A) (I+{\Delta t}F))^n$ and $(R(\Delta t A_h) P_h (I+{\Delta t}F))^n$ with approximation orders $\rho \in \N$ in time and $\kappa \in \{2,3\}$ in space. For each $\alpha \in [0,r]$, there is a constant $C<\infty$ such that, for all $h, \Delta t \in (0,1]$,
		\begin{align*}
		\sup_{t \in [0,T]} \norm{P_1(\hat{S}_{h, \Delta t}(t)-\hat{S}_{\Delta t}(t))\Theta^{-\frac{\alpha}{2}}}{\cL(\cH,H)} \le C \left({\Delta t}^{\min(\alpha \frac{\rho}{\rho+1},1)} + h^{\min(\alpha \frac{\kappa}{\kappa+1},\kappa)} \right).
		\end{align*}
	\end{lemma}
	By the same Gronwall argument, there is a constant $C<\infty$ such that, for all $h, \Delta t \in (0,1]$ and $t \in [0,T]$, 
	\begin{equation}
	\label{eq:wave_stab_pert_semi_semigroup}
	\norm{\hat{S}_{\Delta t}(t)}{\cL(\cH)} \le C \text{ and } \norm{\hat{S}_{h,\Delta t}(t)}{\cL(\cH)} \le C.
	\end{equation}
	
	We can now define the fully discrete approximation of~\eqref{eq:C} and prove an error estimate with respect to the first component $\Cov(U(t))$, $t \in [0,T]$. It is denoted by $\tilde K_{h,\Delta t}$ and defined by $\tilde K_{h,\Delta t}(t_0) = P_h Q_\xi P_h$ and, for  $j \ge 1$, $\tilde K_{h,\Delta t}(t_j) = \hat{S}_{h,\Delta t} \tilde K_{h,\Delta t}(t_{j-1}) \hat{S}_{h,\Delta t}^* + \Delta t P_h B (P_h B)^*$.
	In closed form, 
	\begin{equation*}
	\begin{split}
	\tilde K_{h,\Delta t}(t_n) &= \hat{S}_{h,\Delta t}^{n} Q_\xi (\hat{S}_{h,\Delta t}^{n})^* + \Delta t \sum^{n-1}_{j=0} \hat{S}_{h,\Delta t}^{n-j-1} B (\hat{S}_{h,\Delta t}^{n-j-1} B)^* \\ 
	&= \hat{S}_{h,\Delta t}^{n} Q_\xi (\hat{S}_{h,\Delta t}^{n})^* + \Delta t \sum^{n-1}_{j=0} \hat{S}_{h,\Delta t}^{n-j-1} B Q B^* (\hat{S}_{h,\Delta t}^{n-j-1})^*,
	\end{split}
	\end{equation*}
	where the adjoint of $B$ is taken with respect to $\cL(H,\cH)$ in the last expression. In Theorem~\ref{thm:wavefullydiscreteerror}, the error estimate for $\tilde K_{h,\Delta t}$ is given. %
	
	\begin{theorem}
		\label{thm:wavefullydiscreteerror}
		Let Assumptions~\ref{assumptions:waveexistence} and~\ref{assumptions:waveregularity} be satisfied  for some $r \ge 0$ and $i \in \{1,2\}$ and let $\tilde K_{h, \Delta t}$ be based on the perturbed semigroup approximation $\hat{S}_{h, \Delta t}$ with order $\rho \in \N$ in time and $\kappa \in \{2,3\}$ in space. Then, there is a constant $C<\infty$ such that for all $h, \Delta t \in (0,1]$
		\begin{equation*}
		\sup_{n \in \{0,\ldots,N_{\Delta t}\}} \norm{P_1(K(t_n)-\tilde K_{h, \Delta t}(t_n))P_1^*}{\cL_i(\cH)} \le C \left({\Delta t}^{\min(r \frac{\rho}{\rho+1},1)} + h^{\min(r \frac{\kappa}{\kappa+1},\kappa)} \right).
		\end{equation*}
	\end{theorem}
	\begin{proof}
		First, we recall that $\hat{S}_{\Delta t} = \tilde{S}_{\Delta t}(I + \Delta t F)$ so that we may rewrite~\eqref{eq:wave_semi_rec} as
		\begin{align*}
		\tilde K_{\Delta t}(t_j) &=  (\tilde S_{\Delta t} + \Delta t \tilde S_{\Delta t} F) \tilde K_{\Delta t}(t_{j-1}) (\tilde S_{\Delta t} + \Delta t \tilde S_{\Delta t} F)^* - {\Delta t}^2 \tilde S_{\Delta t} F \tilde K_{\Delta t}(t_{j-1}) F^* \tilde S_{\Delta t}^* \\
		&\quad+ \Delta t \tilde S_{\Delta t} B Q B^* \tilde S_{\Delta t}^*\\
		&=\hat{S}_{\Delta t} \tilde K_{\Delta t}(t_{j-1}) \hat{S}_{\Delta t}^* - {\Delta t}^2 \tilde S_{\Delta t} F \tilde K_{\Delta t}(t_{j-1}) F^* \tilde S_{\Delta t}^* + \Delta t \tilde S_{\Delta t} B Q B^* \tilde S_{\Delta t}^*
		\end{align*}
		Iterating this equality and making use of~\eqref{eq:operator-identity} yields, for $n = 1, \ldots, N_{\Delta t}$,
		\begin{align*}
		\tilde K_{\Delta t}(t_n) = \hat{S}_{\Delta t}^{n} Q_\xi (\hat{S}_{\Delta t}^{n})^*  + \Delta t \sum^{n-1}_{j=0} \hat{S}_{\Delta t}^{n-j-1} B Q B^* (\hat{S}_{\Delta t}^{n-j-1})^* + R_{\Delta t}(t_n),
		\end{align*}
		where 
		\begin{align*}
		R_{\Delta t}(t_n)  &= \frac{\Delta t}{2} \sum_{j=0}^{n-1} \hat{S}_{\Delta t}^{n-j-1} (\tilde S_{\Delta t} - I) (B Q B^* - {\Delta t}F \tilde K_{\Delta t}(t_j) F^*)  (\tilde S_{\Delta t} + I)^* (\hat{S}_{\Delta t}^{n-j-1})^* \\
		&\quad+ \frac{\Delta t}{2} \sum_{j=0}^{n-1} \hat{S}_{\Delta t}^{n-j-1} (\tilde S_{\Delta t} + I) (B Q B^* - {\Delta t}F \tilde K_{\Delta t}(t_j) F^*)  (\tilde S_{\Delta t} - I)^* (\hat{S}_{\Delta t}^{n-j-1})^* \\
		&\quad- {\Delta t}^2 \sum_{j=0}^{n-1} \hat{S}_{\Delta t}^{n-j-1} F \tilde K_{\Delta t}(t_j) F^*  (\hat{S}_{\Delta t}^{n-j-1})^* = \mathrm{I} + \mathrm{II} + \mathrm{III}.
		\end{align*}
		Adding and subtracting $\tilde K_{\Delta t}(t_n)$ in $K(t_n)-\tilde K_{h, \Delta t}(t_n)$, we therefore obtain the split
		\begin{align*}
		&\norm{P_1(K(t_n)-\tilde K_{h, \Delta t}(t_n))P_1^*}{\cL_i(H)} \\ &\quad\le\norm{P_1(K(t_n)-\tilde K_{\Delta t}(t_n))P_1^*}{\cL_i(H)} + \norm{P_1(\hat{S}_{h, \Delta t}^{n} - \hat{S}_{\Delta t}^{n}) Q_\xi (\hat{S}_{h, \Delta t}^{n} + \hat{S}_{\Delta t}^{n})^*P_1^*}{\cL_i(H)} \\
		&\qquad+ \lrnorm{\Delta t \sum^{n-1}_{j=0} P_1 (\hat{S}_{h, \Delta t}^{n-j-1} - \hat{S}_{\Delta t}^{n-j-1}) B Q B^* (\hat{S}_{h, \Delta t}^{n-j-1} + \hat{S}_{\Delta t}^{n-j-1})^* P_1^*}{\cL_i(H)} \\
		&\qquad+ \norm{P_1\mathrm{I}P_1^*}{\cL_i(H)} + \norm{P_1\mathrm{II}P_1^*}{\cL_i(H)} + \norm{P_1\mathrm{III}P_1^*}{\cL_i(H)}.
		\end{align*}
		The first term of this split is treated by Theorem~\ref{thm:wavesemidiscreteerror}, noting that $P_1\in \cL(\cH,H)$. For the second term, we use Lemma~\ref{lem:wave_fullydisc_semidisc_error}, Assumption~\ref{assumptions:waveregularity}\ref{assumptions:waveregularity:initialcov} and~\eqref{eq:wave_stab_pert_semi_semigroup} to see that
		\begin{align*}
		&\norm{P_1(\hat{S}_{h, \Delta t}^{n} - \hat{S}_{\Delta t}^{n}) Q_\xi (\hat{S}_{h, \Delta t}^{n} + \hat{S}_{\Delta t}^{n})^*P_1^*}{\cL_i(\cH)} \\
		&\quad\lesssim \norm{P_1(\hat{S}_{h, \Delta t}^{n} - \hat{S}_{\Delta t}^{n})\Theta^{-\frac{r}{2}}}{\cL(\cH,H)} \norm{\Theta^{\frac{r}{2}}Q_\xi}{\cL_i(\cH)} \lesssim {\Delta t}^{\min(r \frac{\rho}{\rho+1},1)} + h^{\min(r \frac{\kappa}{\kappa+1},\kappa)}.
		\end{align*}
		For the third term, the same estimates combined with Assumption~\ref{assumptions:waveregularity}\ref{assumptions:waveregularity:Q} and~\eqref{eq:B_bound_2} yield the bound
		\begin{align*}
		\Delta t \sum^{n-1}_{j=0} \norm{P_1 (\hat{S}_{h, \Delta t}^{n-j-1} - \hat{S}_{\Delta t}^{n-j-1}) \Theta^{-\frac{r}{2}} }{\cL(\cH,H)} \norm{\Lambda^{\frac{r-1}{2}} Q \Lambda^{-\frac{1}{2}}}{\cL_i(H)}  \lesssim {\Delta t}^{\min(r \frac{\rho}{\rho+1},1)} + h^{\min(r \frac{\kappa}{\kappa+1},\kappa)}.
		\end{align*}
		For $\norm{P_1\mathrm{I}P_1^*}{\cL_i(H)}$ and $\norm{P_1\mathrm{II}P_1^*}{\cL_i(H)}$, there is a constant $C<\infty$ such that for all $\Delta t \in (0,1]$
		\begin{equation*}
		\norm{(\tilde S_{\Delta t} - I)\Theta^{-\frac{r}{2}}}{\cL(\cH)} \le \norm{(\tilde S_{\Delta t} - S(\Delta t))\Theta^{-\frac{r}{2}}}{\cL(\cH)} + \norm{(S(\Delta t)-S(0))\Theta^{-\frac{r}{2}}}{\cL(\cH)} \le C {\Delta t}^{\min(r \frac{\rho}{\rho+1},1)}.
		\end{equation*}
		Here we have used the H\"older regularity of the semigroup $S$ \cite[Lemma~4.2]{KLL13} and Lemma~\ref{lem:wave_semidisc_semigroup_error}.
		Using this estimate, we obtain for the term $\norm{P_1\mathrm{I}P_1^*}{\cL_i(H)}$ that 
		\begin{align*}
		\norm{P_1\mathrm{I}P_1^*}{\cL_i(H)} &\le \frac{\Delta t}{2} \sum^{n-1}_{j=0} \norm{P_1\hat{S}^{n-j-1}_{\Delta t}}{\cL(\cH,H)}^2 \norm{(\tilde S_{\Delta t} - I)\Theta^{-\frac{r}{2}}}{\cL(\cH)} \norm{B\Lambda^\frac{1}{2}}{\cL(H,\cH)}^2 \norm{\tilde S_{\Delta t} + I}{\cL(\cH)} \\
		&\hspace{0.5em} \times  \big(\norm{\Lambda^\frac{r-1}{2} Q \Lambda^{-\frac{1}{2}}}{\cL_i(H)} + \Delta t \norm{\Lambda^\frac{r-1}{2}G\Lambda^{-\frac{r}{2}}}{\cL(H)}  \norm{\Theta^\frac{r}{2} \tilde K_{\Delta t}(t_j)}{\cL_i(H)} \norm{\Lambda^{-\frac{1}{2}}G}{\cL(H)} \big) \\
		&\lesssim {\Delta t}^{\min(r \frac{\rho}{\rho+1},1)},
		\end{align*}
		with an analogous result for $\norm{P_1\mathrm{II}P_1^*}{\cL_i(H)}$. Here, we also made use of Lemma~\ref{lem:wave-regularity}. This is also used to estimate the final term in the next calculation, completing the proof by Assumption~\ref{assumptions:waveregularity}\ref{assumptions:waveregularity:initialcov} and~\eqref{eq:B_bound_2} via
		\begin{equation*}
		\norm{P_1\mathrm{III}P_1^*}{\cL_i(H)} \le {\Delta t}^2  \sum^{n-1}_{j=0} \norm{P_1\hat{S}^{n-j-1}_{\Delta t}}{\cL(\cH,H)}^2 \norm{\Lambda^{-\frac{1}{2}}G}{\cL(H)}^2 \norm{\tilde K_{\Delta t}(t_j)}{\cL_i(H)} \lesssim \Delta t.
		\qedhere
		\end{equation*}	
	\end{proof}
	
	\begin{example}
		\label{ex:wave}
		Finally, we demonstrate the theoretical results obtained above numerically in the case that $\cD=(0,1), T=1$. We again let $\xi$ be deterministic. We set $G = -Q$ in order to obtain the model suggested in \cite{D09} for the vertical movement of a DNA strand suspended in fluid. We consider the piecewise linear finite element method for our discretization in space and the Crank--Nicolson discretization for our temporal approximation, so that we may take $\rho = \kappa = 2$ in Theorem~\ref{thm:wavefullydiscreteerror}. We let $Q$ be an integral operator with kernel $q$, as in the second part of Example~\ref{ex:heat}.
		
		First, we let $q(x,y) = q(x-y)$ be a Mat\'ern covariance function, which, for $z \in \R$, is given by $q(z) = \sigma^2 2^{1-\nu} \Gamma(\nu)^{-1} (\sqrt{2 \nu} z \rho^{-1})^\nu K_\nu (\sqrt{2 \nu} z \rho^{-1})$
		with parameters $\sigma = 10, \nu = 0.01$ and $\rho = 0.1$. Here $K_\nu$ denotes the modified Bessel function of the second kind. Figure~\ref{subfig:wave-path} shows an approximate realization of the solution to~\eqref{eq:original_wave_equation} for this choice of $q$ and Figure~\ref{subfig:wave-cov} shows its covariance function, corresponding to~$P_1 \tilde K_{h, {\Delta t}}(t)P_1^*$, at $t=0.1$. Since the Fourier transform of $q$ is proportional to $\xi \mapsto (1 + \xi^2)^{-\nu - d/2} = (1 + \xi^2)^{-0.501}$, the results of \cite[Section~4]{KLP21b} implies that $\norm{ \Lambda^{\frac{r-1}{2}} Q}{\cL_2(H)} < \infty$ for $r < 3/2$. Moreover, since $d=1$, $\norm{ \Lambda^{-\frac{1}{2}}}{\cL_2(H)} = \norm{I}{\cL_2(\dot{H}^1,H)} < \infty$.
		By~\eqref{eq:schatten_bound_2}, therefore, Assumption~\ref{assumptions:waveregularity} is satisfied with $r<3/2$ and $i \in \{1,2\}$. From Theorem~\ref{thm:wavefullydiscreteerror} we expect to see a convergence rate of essentially $1$ if we plot the errors $\norm{P_1(K(T)-\tilde K_{h, \Delta t}(T))P_1^*}{\cL_i(H)}$, $i \in \{1,2\}$, with respect to decreasing values of $h = \Delta t$ in a log-log plot. This is in line with our observations in Figure~\ref{subfig:wave-errors-1} which shows the errors for $h = \Delta t = 2^{-1},\ldots, 2^{-8}$. We have again used a reference solution $\tilde K_{h, \Delta t}(T)$ at $h = \Delta t = 2^{-9}$.
		
		\begin{figure}[ht!]
			\centering
			
			\subfigure[Errors with Mat\'ern noise for mesh sizes $h = \Delta t = 2^{-1},\ldots, 2^{-8}$.\label{subfig:wave-errors-1}]{\includegraphics[width = .49\textwidth]{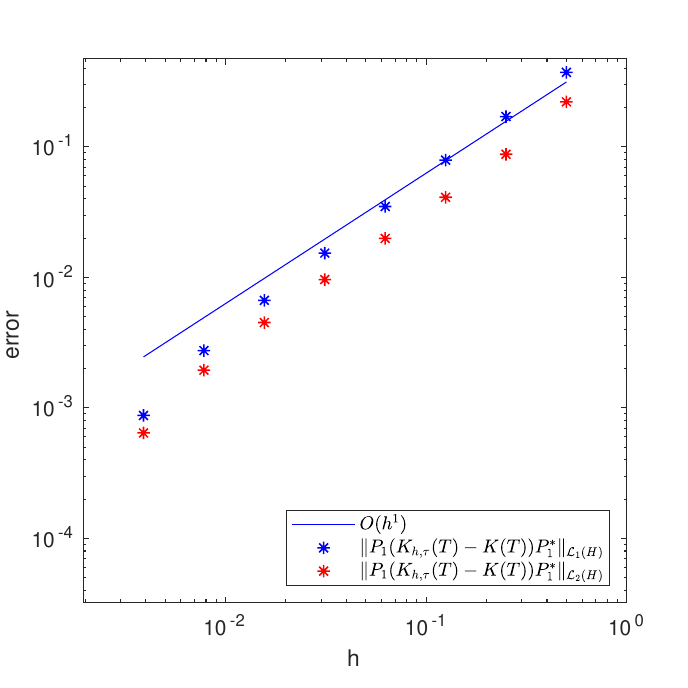}}
			\subfigure[Errors with Brownian bridge noise for mesh sizes $h = \sqrt{\Delta t} = 2^{-1},\ldots, 2^{-5}$. \label{subfig:wave-errors-2}]{\includegraphics[width = .49\textwidth]{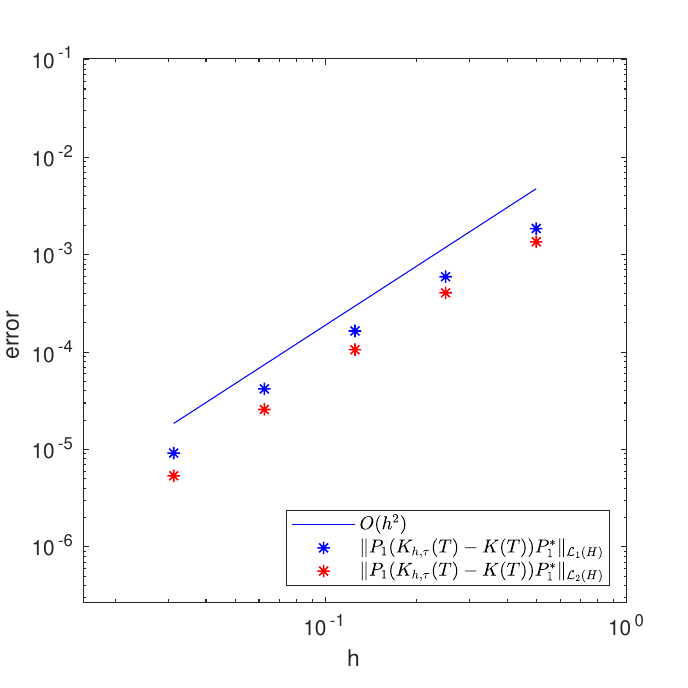}}
			\caption{Approximate errors $\norm{P_1(K(T)-\tilde K_{h, \Delta t}(T))P_1^*}{\cL_i(H)}$, $i \in \{1,2\}$, for the equations of Example~\ref{ex:wave}.}
		\end{figure} 
		
		Next, we set $q(x,y) = \min(x,y)- xy;$ for $x,y \in \cD$, i.e., the covariance function of a Brownian bridge on $\cD$. By its Karhunen--Lo\`eve expansion one obtains $Q = \Lambda^{-1}$. Since $\Lambda^{-1}$ has eigenvalues $(j^2\pi^2)_{j = 1}^\infty$, Assumption~\ref{assumptions:waveregularity} is fulfilled for all $r < 7/2$ when $i=2$ and for all $r < 3$ when $i=1$. In either case, we expect to see a convergence rate of order $2$ if we plot the errors $\norm{P_1(K(T)-\tilde K_{h, \Delta t}(T))P_1^*}{\cL_i(H)}$, $i \in \{1,2\}$, with respect to decreasing values of $h = \sqrt{\Delta t}$ in a log-log plot. This is in line with~Figure~\ref{subfig:wave-errors-2}, where we have plotted the errors with $h = \sqrt{\Delta t} = 2^{-1},\ldots, 2^{-5}$. In place of $K(T)$ we have used $\tilde K_{h, \Delta t}(T)$ at $h = \sqrt{\Delta t} = 2^{-6}$.  
	\end{example}
	
\bibliographystyle{hplain}
\bibliography{spde-fem-covariance}	

\appendix

\section*{Appendix A.\ A derivation of a trace-class error expression}
\label{sec:appendix}

This appendix serves to provide a derivation of the expression 
\begin{equation}
\label{eq:trace-error-formula-app}
\norm{\tilde K_{h, {\Delta t}}(T)-\tilde K_{h', {\Delta t}'}(T)}{\cL_1(H)} = \trace\left(  \left|\mathbf{N}_{h,h'}^{\frac{1}{2}} \left[\begin{array}{cc}
\mathbf{K}_{N_{\Delta t},h,\Delta t} & 0 \\
0 & -\mathbf{K}_{N_{{\Delta t}'},h',{\Delta t}'}
\end{array}\right] 	\mathbf{N}_{h,h'}^{\frac{1}{2}}\right| \right),
\end{equation}
in the setting of Example~\ref{ex:heat}.
First note that $\norm{\tilde K_{h, {\Delta t}}(T)-\tilde K_{h', {\Delta t}'}(T)}{\cL_1(H)} = \trace(|\tilde K_{h, {\Delta t}}(T)-\tilde K_{h', {\Delta t}'}(T)|)$. Next, let us write $\psi_j = \phi^h_{j}$ for $j = 1, \ldots, N_h$, $\psi_j = \phi^{h'}_{j-N_{h}}$ for $j=N_h+1,\ldots,N_h+N_{h'}$ and 
\begin{equation*}
\Upsilon = \sum^{N_h+N_{h'}}_{i,j=1} \Big(\mathbf{N}_{h,h'}^{-\frac{1}{2}}\Big| \mathbf{N}_{h,h'}^{\frac{1}{2}} \Big[\begin{array}{cc}
\mathbf{K}_{N_{\Delta t},h,\Delta t} & 0 \\
0 & -\mathbf{K}_{N_{{\Delta t}'},h',{\Delta t}'}
\end{array}\Big] \mathbf{N}_{h,h'}^{\frac{1}{2}} \Big| \mathbf{N}_{h,h'}^{-\frac{1}{2}}\Big)_{ij} \psi_i \otimes \psi_j.
\end{equation*}
Here $\mathbf{N}_{h,h'}^{-1/2}$ denotes the pseudoinverse of $\mathbf{N}_{h,h'}^{1/2}$. Note that $(\mathbf{N}_{h,h'})_{ij} = \inpro{\psi_i}{\psi_j}$ so that $\mathbf{N}_{h,h'}, \mathbf{N}_{h,h'}^{1/2}$ and $\mathbf{N}_{h,h'}^{-1/2}$ are in $\Sigma^+(\R^{N_h+N_{h'}})$. We claim that $\Upsilon = |\tilde K_{h, {\Delta t}}(T)-\tilde K_{h', {\Delta t}'}(T)|$. To see this, we first note that $\Upsilon \in \Sigma^+(H)$ as a consequence of the matrix of coefficients in this sum being an element of $\Sigma^+(\R^{N_h+N_{h'}})$. Thus, it suffices to show that $\Upsilon^2 = (\tilde K_{h, {\Delta t}}(T)-\tilde K_{h', {\Delta t}'}(T))^*(\tilde K_{h, {\Delta t}}(T)-\tilde K_{h', {\Delta t}'}(T))$. By a direct calculation using the definition of the tensor product $\otimes$ and symmetry of $\mathbf{N}_{h,h'}$, it follows that 
\begin{equation}
\label{eq:upsilon-sq}
\begin{split}
\Upsilon^2 &= \sum^{N_h+N_{h'}}_{i,j=1} \Big(\mathbf{N}_{h,h'}^{-\frac{1}{2}}\Big| \mathbf{N}_{h,h'}^{\frac{1}{2}} \Big[\begin{array}{cc}
\mathbf{K}_{N_{\Delta t},h,\Delta t} & 0 \\
0 & -\mathbf{K}_{N_{{\Delta t}'},h',{\Delta t}'}
\end{array}\Big] \mathbf{N}_{h,h'}^{\frac{1}{2}} \Big| \\
&\hspace{5em}\times
\mathbf{P}_{\im(\mathbf{N}_{h,h'}^\frac{1}{2})}
\Big| \mathbf{N}_{h,h'}^{\frac{1}{2}} \Big[\begin{array}{cc}
\mathbf{K}_{N_{\Delta t},h,\Delta t} & 0 \\
0 & -\mathbf{K}_{N_{{\Delta t}'},h',{\Delta t}'}
\end{array}\Big] \mathbf{N}_{h,h'}^{\frac{1}{2}} \Big|
\mathbf{N}_{h,h'}^{-\frac{1}{2}}\Big)_{ij} \psi_i \otimes \psi_j,
\end{split}
\end{equation}
where $\mathbf{P}_{\im(\mathbf{N}_{h,h'}^{1/2})} = \mathbf{N}_{h,h'}^{-1/2} \mathbf{N}_{h,h'}^{1/2} = \mathbf{N}_{h,h'}^{1/2} \mathbf{N}_{h,h'}^{-1/2}$ is the projection onto $\im(\mathbf{N}_{h,h'}^{1/2})$, the range of $\mathbf{N}_{h,h'}^{1/2}$. Since the kernels, and hence the ranges, of a matrix in $\Sigma^+(\R^{N_h+N_{h'}})$ and its symmetric positive semidefinite square root coincide, we have $\im(\mathbf{N}_{h,h'}) = \im(\mathbf{N}_{h,h'}^{1/2})$ and thus
\begin{equation}
\label{eq:range-statement}
\begin{split}
&\im\Big(\mathbf{N}_{h,h'}\Big) \\
&\quad\supset \im\Big(\mathbf{N}_{h,h'}^{\frac{1}{2}} \Big[\begin{array}{cc}
\mathbf{K}_{N_{\Delta t},h,\Delta t} & 0 \\
0 & -\mathbf{K}_{N_{{\Delta t}'},h',{\Delta t}'}
\end{array}\Big]\mathbf{N}_{h,h'}\Big[\begin{array}{cc}
\mathbf{K}_{N_{\Delta t},h,\Delta t} & 0 \\
0 & -\mathbf{K}_{N_{{\Delta t}'},h',{\Delta t}'}
\end{array}\Big]\mathbf{N}_{h,h'}^{\frac{1}{2}}\Big) \\
&\quad=\im\Big(\Big| \mathbf{N}_{h,h'}^{\frac{1}{2}} \Big[\begin{array}{cc}
\mathbf{K}_{N_{\Delta t},h,\Delta t} & 0 \\
0 & -\mathbf{K}_{N_{{\Delta t}'},h',{\Delta t}'}
\end{array}\Big] \mathbf{N}_{h,h'}^{\frac{1}{2}} \Big|\Big).
\end{split}
\end{equation}
Combining this with~\eqref{eq:upsilon-sq}, we find that 
$$\Upsilon^2 = \sum^{N_h+N_{h'}}_{i,j=1} \mathbf{U}_{ij} \psi_i \otimes \psi_j,$$ 
where $\mathbf{U} = \mathbf{P}_{\im(\mathbf{N}_{h,h'})}  \mathbf{\tilde K} \mathbf{P}_{\im(\mathbf{N}_{h,h'})}$ and
\begin{equation*}
\mathbf{\tilde K} = \Big[\begin{array}{cc}
\mathbf{K}_{N_{\Delta t},h,\Delta t} & 0 \\
0 & -\mathbf{K}_{N_{{\Delta t}'},h',{\Delta t}'}
\end{array}\Big] \mathbf{N}_{h,h'}
\Big[\begin{array}{cc}
\mathbf{K}_{N_{\Delta t},h,\Delta t} & 0 \\
0 & -\mathbf{K}_{N_{{\Delta t}'},h',{\Delta t}'}
\end{array}\Big].
\end{equation*}
Moreover, since 
\begin{equation*}
\tilde K_{h, {\Delta t}}(T)-\tilde K_{h', {\Delta t}'}(T) = \sum^{N_h+N_{h'}}_{i,j=1} \Big(\Big[\begin{array}{cc}
\mathbf{K}_{N_{\Delta t},h,\Delta t} & 0 \\
0 & -\mathbf{K}_{N_{{\Delta t}'},h',{\Delta t}'}
\end{array}\Big]\Big)_{ij}\psi_i \otimes \psi_j
\end{equation*}
it follows from a direct calculation that 
\begin{equation*}
(\tilde K_{h, {\Delta t}}(T)-\tilde K_{h', {\Delta t}'}(T))^*(\tilde K_{h, {\Delta t}}(T)-\tilde K_{h', {\Delta t}'}(T)) = \sum^{N_h+N_{h'}}_{i,j=1} \mathbf{\tilde K}_{ij} \psi_i \otimes \psi_j.
\end{equation*}
Note now that 
\begin{align*}
&\norm{(\tilde K_{h, {\Delta t}}(T)-\tilde K_{h', {\Delta t}'}(T))^*(\tilde K_{h, {\Delta t}}(T)-\tilde K_{h', {\Delta t}'}(T)) - \Upsilon^2}{\cL_2(H)}^2 \\
&\quad= \sum_{i=1}^\infty \norm{\big((\tilde K_{h, {\Delta t}}(T)-\tilde K_{h', {\Delta t}'}(T))^*(\tilde K_{h, {\Delta t}}(T)-\tilde K_{h', {\Delta t}'}(T)) - \Upsilon^2\big)e_i}{H}^2 \\
&\quad= \sum_{i=1}^\infty \norm{\sum^{N_h+N_{h'}}_{j,k=1} (\mathbf{\tilde K}_{jk} - \mathbf{U}_{jk}) \inpro[H]{\psi_k}{e_i} \psi_j }{H}^2 \\
&\quad= \sum^{N_h+N_{h'}}_{j,k,\ell,m=1} \sum_{i=1}^\infty (\mathbf{\tilde K}_{jk} - \mathbf{U}_{jk}) \inpro[H]{\psi_k}{e_i} (\mathbf{\tilde K}_{\ell m} - \mathbf{U}_{\ell m}) \inpro[H]{\psi_m}{e_i} \inpro[H]{\psi_j}{\psi_\ell} \\
&\quad= \sum^{N_h+N_{h'}}_{j,k,\ell,m=1} (\mathbf{\tilde K}_{jk} - \mathbf{U}_{jk}) \inpro[H]{\psi_k}{\psi_m} (\mathbf{\tilde K}_{\ell m} - \mathbf{U}_{\ell m}) \inpro[H]{\psi_j}{\psi_\ell} \\
&\quad= \sum^{N_h+N_{h'}}_{\ell,m=1} \left( \mathbf{N}_{h,h'} \left( \mathbf{\tilde K} - \mathbf{P}_{\im(\mathbf{N}_{h,h'})}  \mathbf{\tilde K} \mathbf{P}_{\im(\mathbf{N}_{h,h'})} \right) \mathbf{N}_{h,h'} \right)_{\ell m} (\mathbf{\tilde K}_{\ell m} - \mathbf{U}_{\ell m})= 0,
\end{align*}
where the last equality follows from the fact that
\begin{align*}
\mathbf{N}_{h,h'} \left( \mathbf{\tilde K} - \mathbf{P}_{\im(\mathbf{N}_{h,h'})}  \mathbf{\tilde K} \mathbf{P}_{\im(\mathbf{N}_{h,h'})} \right) \mathbf{N}_{h,h'} &= \mathbf{N}_{h,h'} \left( \mathbf{\tilde K} - \mathbf{P}_{\im(\mathbf{N}_{h,h'})}  \mathbf{\tilde K} \right) \mathbf{N}_{h,h'} \\ &= \mathbf{N}_{h,h'} (\mathbf{I} - \mathbf{P}_{\im(\mathbf{N}_{h,h'})}) \mathbf{\tilde K}\mathbf{N}_{h,h'} \\
&= \mathbf{N}_{h,h'}\mathbf{P}_{\ke(\mathbf{N}_{h,h'})}\mathbf{\tilde K}\mathbf{N}_{h,h'} = 0.
\end{align*}
With this we have established that $\Upsilon = |\tilde K_{h, {\Delta t}}(T)-\tilde K_{h', {\Delta t}'}(T)|$. Finally, we obtain~\eqref{eq:trace-error-formula-app} via
\begin{align*}
&\trace(|\tilde K_{h, {\Delta t}}(T)-\tilde K_{h', {\Delta t}'}(T)|) \\
&\quad= \sum^{N_h+N_{h'}}_{i,j=1} \Big(\mathbf{N}_{h,h'}^{-\frac{1}{2}}\Big| \mathbf{N}_{h,h'}^{\frac{1}{2}} \Big[\begin{array}{cc}
\mathbf{K}_{N_{\Delta t},h,\Delta t} & 0 \\
0 & -\mathbf{K}_{N_{{\Delta t}'},h',{\Delta t}'}
\end{array}\Big] \mathbf{N}_{h,h'}^{\frac{1}{2}} \Big| \mathbf{N}_{h,h'}^{-\frac{1}{2}}\Big)_{ij} \inpro[H]{\psi_j}{\psi_i} \\
&\quad= \trace\Big(\mathbf{N}_{h,h'}^{-\frac{1}{2}}\Big| \mathbf{N}_{h,h'}^{\frac{1}{2}} \Big[\begin{array}{cc}
\mathbf{K}_{N_{\Delta t},h,\Delta t} & 0 \\
0 & -\mathbf{K}_{N_{{\Delta t}'},h',{\Delta t}'}
\end{array}\Big] \mathbf{N}_{h,h'}^{\frac{1}{2}} \Big| \mathbf{N}_{h,h'}^{\frac{1}{2}}\Big) \\
&\quad= \trace\Big(\mathbf{P}_{\im(\mathbf{N}_{h,h'}^{1/2})}\Big| \mathbf{N}_{h,h'}^{\frac{1}{2}} \Big[\begin{array}{cc}
\mathbf{K}_{N_{\Delta t},h,\Delta t} & 0 \\
0 & -\mathbf{K}_{N_{{\Delta t}'},h',{\Delta t}'}
\end{array}\Big] \mathbf{N}_{h,h'}^{\frac{1}{2}} \Big|\Big) \\
&\quad= \trace\Big(\Big| \mathbf{N}_{h,h'}^{\frac{1}{2}} \Big[\begin{array}{cc}
\mathbf{K}_{N_{\Delta t},h,\Delta t} & 0 \\
0 & -\mathbf{K}_{N_{{\Delta t}'},h',{\Delta t}'}
\end{array}\Big] \mathbf{N}_{h,h'}^{\frac{1}{2}} \Big|\Big).
\end{align*}
Here we made use of the cyclic property of the trace and~\eqref{eq:range-statement}.
	
\end{document}